\documentclass[11pt,reqno]{amsart}
\usepackage[utf8]{inputenc}

\usepackage{amsmath,amssymb,amsthm,mathrsfs, color}
\usepackage{a4wide}
\usepackage{graphics}
\usepackage{epsfig}
\usepackage{enumitem}

\newtheorem{theorem}{Theorem}[section]
\newtheorem{prop}{Proposition}[section]
\newtheorem{lemma}{Lemma}[section]
\newtheorem{corollary}{Corollary}[section]

\newtheorem{remark}{Remark}[section]

\newcommand{\D}{\mathcal{D}}
\newcommand{\R}{{\mathbb R}}
\newcommand{\E}{{\mathbb E}}

\newcommand*{\tp}{^\top}
\newcommand*{\ta}{^{(1)}}
\newcommand*{\tb}{^{(2)}}
\newcommand*{\taa}{_{(1)}}
\newcommand*{\tba}{_{(2)}}

\def\RR{\mathbb   R}
\def\EE{\mathbb   E}
\def\NN{\mathbb   N}
\def\PP{\mathbb   P}

\def\1{{\mathbf 1}}
\def\2{{\mathbf 2}}
\DeclareMathOperator{\Tr}{Tr}

\makeatletter
 \@addtoreset{equation}{section}
 \makeatother
 
\allowdisplaybreaks
\usepackage{appendix}
\usepackage{hyperref}

\title[Non-asymptotic error bounds for The Multilevel  Monte Carlo Euler method ]{Non-asymptotic error bounds for The Multilevel  Monte Carlo Euler method applied to SDE{s} with constant diffusion coefficient}
\date{\today}
\author{Benjamin Jourdain}
\thanks{This work benefited from the support of the ``chaire Risques financiers'', fondation du risque and  the French National Research Agency
under the program ANR-12-BS01-0019 (STAB)}
\address{Benjamin Jourdain, Universit\'e  Paris-Est, Cermics (ENPC), INRIA, F-77455, Marne-la-Vall\'ee, France }
\email{benjamin.jourdain@enpc.fr}

\author{Ahmed Kebaier}
\thanks{}
\address{Ahmed Kebaier, Universit\'e Paris 13, Sorbonne Paris Cit\'e, LAGA, CNRS, (UMR 7539), F-93430 Villetaneuse, France}
\email{kebaier@math.univ-paris13.fr}


\begin{document}

\maketitle
%%%%%%%%%%%%%%%%%%%%%%%%%%%%%%%%%%%%%%%%%%%%%%%%%%%%%%%%%%%
\begin{abstract}
In this paper, we are interested in deriving non-asymptotic error bounds for the multilevel Monte Carlo method. As a first step, we deal with the explicit Euler discretization of stochastic differential equations with a constant diffusion coefficient. We prove that, as long as the deviation is below an explicit threshold, a Gaussian-type concentration inequality optimal in terms of the variance holds for the multilevel estimator. To do so, we use the Clark-Ocone representation formula and derive bounds for the moment generating functions of the squared difference between a crude Euler scheme and a finer one and of the squared difference of their Malliavin derivatives.
\end{abstract}
%%%%%%%%%%%%%%%%%%%%%%%%%%%%%%%%%%%%%%%%%%%%%%%%%%%%%%%%%%%
\section{introduction}
We are interested in deriving non asymptotic error estimations for the multilevel Monte Carlo estimators introduced by Giles \cite{Gil}. In this paper, as a first step, we deal with estimators of $\E\left[f(X_T)\right]$ where $f:\R^d\to\R$ is Lipschitz continuous with constant $[\dot f]_\infty$,  $T\in(0,+\infty)$ is a deterministic time horizon and $X:=\displaystyle\left(X_t\right)_{\scriptstyle{0\leq t\leq T}}$ is the $\mathbb R^d$-valued solution to the stochastic differential equation with additive noise \begin{equation}\label{1}
dX_{t} = b(X_{t})dt + dW_{t},\;\;\; X_0=x_0\in \mathbb R^d,
\end{equation}
%\begin{equation}\label{1}
%dX_{t} = b(X_{t})dt + \sum_{j=1}^q {{\sigma_j(X_{t})}}dW^j_{t},\;\;\; X_0=x\in \mathbb R^d
%\end{equation}
driven by the $d$-dimensional Brownian motion  $\displaystyle W=(W^1,\dots,W^d)$ and with Lipschitz drift function $b:\R^d\to\R^d$ :
$$
\exists C_{b}<+\infty,\;\forall x,y\in\mathbb R^d\quad |b(x)-b(y)|\leq C_{b}|x-y|.\leqno(\mathcal H_{GL})
$$ When $d=1$, this additive noise setting is not restrictive. Indeed any stochastic differential equation $dY_t=\sigma(Y_t)dW_t+\eta(Y_t)dt$ with multiplicative noise given by some function $\sigma :\R\to\R_+^*$ such that $\frac{1}{\sigma}$ is locally integrable can be reduced to \eqref{1} by the Lamperti transformation : for $\varphi(y)=\int_{y_0}^y\frac{dz}{\sigma(z)}$, $X_t=\varphi(Y_t)$ solves \eqref{1} with $b(x)=\left(\frac{\eta}{\sigma}-\frac{\sigma'}{2}\right)(\varphi^{-1}(x))$.

For $n\in{\mathbb N}^*$, we consider the simple Euler-Maruyama approximation $X^n$ with time step $T/n$ and we introduce its continuous version given by
\begin{equation}\label{euler}
dX^n_t=b(X_{\eta_n(t)})dt+dW_t,\;\;\;\eta_n(t)=\left\lfloor \frac{nt}{T}\right\rfloor\frac{T}{n},\;\;X^n_0=x_0.
\end{equation}
When $b$ is smooth, both the strong and the weak errors of this scheme converge to $0$ with order $1$ as $n\to\infty$. According to \cite{Gil}, the complexity for the multilevel Monte Carlo estimator of $\E\left[f(X_T)\right]$ based on this scheme to achieve a root mean square error $\varepsilon$ is ${\mathcal O}(\varepsilon^{-2})$ in the limit $\varepsilon\to 0$, the same as in a standard Monte Carlo method with i.i.d. unbiased samples. For positive integers $m$ and $L$ and $(N_\ell)_{0\le \ell\le L}$, the Multilevel Monte Carlo method approximates the expectation of interest $\EE\left[f( X_T)\right]$  by 
\begin{equation}
\displaystyle\ \hat Q=\frac{1}{N_0}\sum_{k=0}^{N_0}f( X^1_{T,k})+\sum_{\ell=0}^{L}\frac{1}{N_\ell}
\sum_{k=1}^{N_\ell}\left(f( X^{m^{\ell}}_{T,k})-f( X^{m^{\ell-1}}_{T,k})\right).\label{estimult}
\end{equation}
The processes  $((X^{m^{\ell}}_{t,k})_{0\leq t\leq T})_k$ denote independent copies of the Euler scheme with time step $m^{-\ell}T$ for $\ell\in\{0,\cdots,L\}$.
Here, it is important to point out that all these $L+1$ Monte
Carlo estimators have to be based on different, independent samples.  However, for fixed $k$ and $\ell$, 
the simulations $f( X^{m^{\ell}}_{T,k})$ and $f( X^{m^{\ell-1}}_{T,k})$ 
have to be based on the same Brownian path but with different times steps $m^{-\ell}T$ and $m^{-(\ell-1)}T$. 

Our main motivation is the derivation of Gaussian type concentration inequalities for $\hat{Q}-\E[f(X_T)]$, a natural question, which, to our knowledge has not been addressed in the literature. Frikha and Menozzi \cite{FriMen} obtained concentration inequalities for $f(X^n_T)-\E[f(X^n_T)]$. Deriving estimations of the moment generating function of the differences $f(X^{mn}_T)-f(X^{n}_T)-\E\left[f(X^{mn}_T)-f(X^{n}_T)\right]$ which are optimal in terms of their variances is a much more delicate task and adapting their approach seems to be problematic. However, the boundedness of the Malliavin derivatives $\D X^{n}_T$ and $\D X^{mn}_T$ in the additive noise setting permits to follow the approach of Houdr\'e and Privault \cite{HouPri} based on the Clark-Ocone formula and this is one reason why we focus on this setting. Another reason is that for stochastic differential equations with multiplicative noise, more sophisticated schemes, like the Milstein scheme in the commutative case or the Giles and Szpruch  \cite{GilSzp} scheme in the general case, are necessary to improve to two the order one of convergence of the variance of $\left(f( X^{m^{\ell}}_{T})-f( X^{m^{\ell-1}}_{T})\right)^2$ and recover the unbiased Monte Carlo complexity. 

In Section \ref{GF}, when $b$ is $\mathscr C^2$, Lipschitz continuous and the Laplacians of its coordinates have an affine growth, we first derive non-asymptotic estimates of the squared error $\E[(\hat{Q}-\E[f(X_T)])^2]$ of the multilevel Monte Carlo estimator (MLMC) \eqref{estimult} for a  Lipschitz continuous test function $f$ by computing explicit bounds for the bias $\E[f(X^{m^L}_T)-f(X_T)]$ and variance ${\rm Var}[f( X^{m^{\ell}}_{T})-f( X^{m^{\ell-1}}_{T})]$. Then we optimize the parameters $(L,(N_\ell)_{0\le \ell \le L})$ in order to minimize the computation cost needed to achieve a root mean square error smaller than a given precision $\varepsilon$. It turns out that, as $\varepsilon\to 0$, the optimal bias is of order ${\mathcal O}(\varepsilon^{4/3})$, which, to our knowledge, has not been pointed out in the MLMC literature so far. Notice that, for stochastic differential equations with a non constant diffusion coefficient (multiplicative noise), this property remains true for the multilevel Monte Carlo estimator based on the Giles and Szpruch scheme \cite{GilSzp}, since it exhibits the same orders of convergence of the bias and the variance within a given level as \eqref{estimult}.

In Section \ref{so}, we state and derive our main result : as long as the deviation is below an explicit threshold, a Gaussian-type concentration inequality optimal in terms of the variance holds for the multilevel estimator \eqref{estimult}. 
Denoting by $\hat Q_{\varepsilon}$ the multilevel Monte Carlo estimator corresponding to the optimal choice of parameters discussed in Section 2, we obtain the existence of explicit  positive constants $c_1,c_2$ and $c_3$ such that
\begin{equation}
   \forall \varepsilon\in (0,c_1),\;\forall \alpha\in(0,c_2\varepsilon^{2/3}),\;\PP\left(|\hat Q_\varepsilon-\mathbb Ef(X_T)|\ge\alpha\right)\le 2e^{\frac{2}{c_3}}e^{-\frac{\alpha^2}{c_3\varepsilon^2}}.\label{concbound}
\end{equation}
In view of the last factor, this bound is optimal in terms of the precision $\varepsilon$ (up to the value of the multiplicative constant $c_3$). For deviations $\alpha(\varepsilon)$ depending on $\varepsilon$ and such that $\lim_{\varepsilon\to 0}\frac{\alpha(\varepsilon)}{\varepsilon}=\infty$, the right-hand side of \eqref{concbound} converges to $0$ far quicker than the one of the bound $\PP\left(|\hat Q_{\varepsilon}-\mathbb Ef(X_T)|\ge\alpha\right)\le \frac{\varepsilon^2}{\alpha^2}$ consequence of the Markov inequality.
We show in Corollary \ref{cordev2} that the same inequality holds for deviations $\alpha$ up to the order $\ln(1/\varepsilon)^{-1/\beta}$ with $\beta>1$ for a multilevel estimator with increased numbers of simulations in the high levels but with computation cost still of order ${\mathcal O}(\varepsilon^{-2})$ as $\varepsilon\to 0$. 
Moreover, we derive a comparison between the root mean square  error (RMSE) and Orlicz norm for both standard and multilevel Monte Carlo. 
It turns out that compared to standard Monte Carlo, the MLMC estimator achieves the same complexity reduction for Orlicz norm as for the RMSE (see Section \ref{sec:orlicz}). 
The limitation, mentioned  above, on the range of deviations $\alpha$ for which the Gaussian-type concentration inequality holds is related to a corresponding limitation on the range of parameters for which we are able to estimate (optimally in terms of the variance) the moment generating function of $\hat Q-\E f(X_T)$. This comes from the quadratic contributions of the Brownian increments that one obtains when applying Itô's formula twice to exhibit the order of the difference $f(X^{mn}_T)-f(X^n_T)$ for $n\in\{1,m,\hdots,m^{L-1}\}$. Maybe these restrictions could be relaxed when replacing the Brownian increments in the Euler schemes by Rademacher random variables like in the weak MLMC method introduced by Belomestny and Nagapetyan \cite{BeloNaga}. Nonetheless the derivation of concentration bounds for the weak MLMC estimators would require a different approach.
Indeed, we use the Clark-Ocone formula as suggested in Houdr\'e and Privault  \cite{HouPri}, to relate the estimation of the moment generating function of $f(X^{mn}_T)-f(X^{n}_T)-\E\left[f(X^{mn}_T)-f(X^{n}_T)\right]$ for $n\in\{1,m,\hdots,m^{L-1}\}$, to the ones of the squared difference between the crude Euler scheme with $n$ steps and the finer one with $mn$ steps and of the squared difference of their Malliavin derivatives. Such estimations are respectively proved in sections 4 and 5 by using a clever decomposition of the difference between the two schemes. They are combined in Section 6 to estimate the moment generating function of $\hat Q-\E f(X_T)$.
\subsection*{Notations} 
Throughout this paper, we shall use the following  notations. 
\begin{itemize}
\item We denote by $\mathscr C^{\infty}_p(\RR^d,\RR^q)$  the set of all infinitely differentiable functions $g : \RR^d \rightarrow \RR^q$ such that $g$ and all of its partial derivatives have at most polynomial growth.
\item For $n\in\NN^*$, we denote by $\mathscr C^{n}(\RR^d,\RR^q)$  the set of all $n$ times continuously differentiable functions $g : \RR^d \rightarrow \RR^q$.
\item For $g:\RR^d\rightarrow \RR^d$,  we denote by $\nabla g$ the Jacobian matrix defined for all $i,j\in\{1,\hdots,d\}$ and $x\in\R^d$ by $(\nabla g)_{ij}(x)
=\partial_{x_j}g_i(x)$. 
\item For $g:\RR^d\rightarrow \RR^d$,  $\Delta g:\RR^d\rightarrow \RR^d$ denotes the function obtained by applying the Laplacian to each coordinate of $g$.
\item The ceiling function and floor function are denoted respectively by $\lceil\cdot\rceil$ and $\lfloor\cdot\rfloor$ (i.e. for  $x\in\RR$, $\lceil x \rceil$
represents the smallest integer no less than $x$; $\lfloor x\rfloor$  represents the largest integer no greater than
$x$).
\item For $d\in\NN^*$, we denote by $\mathcal M_d$ the set of
real $d$-square matrices with identity matrix $I_d$.  
\item For any matrix sequence $(A_k)_{k\in\NN}\in \mathcal M_d$,  we use the following convention
\begin{equation*}
\prod_{k=n_2}^{n_1}A_k=A_{n_2}\cdots A_{n_1},\; \; \forall  n_1,n_2\in\NN \mbox{ s.t. }  n_1\le n_2.
\end{equation*}
\item The Euclidean  inner product and the associated norm are respectively denoted by  $\cdot$ and $|\cdot|$.
\item For  $M\in\mathcal M_d$, the matrix norm induced by the Euclidean norm $|\cdot|$ is denoted by 
$$\|M\|=\sup_{x\in\R^d:|x|=1}|Mx|.$$
\item For any adapted $\R^d$-valued process $(H_t)_{0\leq t\leq T}$ and $\mathcal M_d$-valued process $(M(t))_{0\leq t\leq T}$, we denote $$|H|:=\left\|\left(\int_0^T|H_t|^2dt\right)^{1/2}\right\|_\infty\quad\mbox{and}\quad |M|:=\left\|\left(\int_0^T\Tr\left[M(t)M(t)\tp\right]dt\right)^{1/2}\right\|_\infty$$
where for $A\in\mathcal M_d$, $A\tp$ and $\Tr[A]$ denote respectively the transpose and the trace of matrix $A$.  
\end{itemize}

%%%%%%%%%%%%%%%%%%%%%%%%%%%%%%%%%%%%%%%%%%%%%%%%%%%%%%%%%%%

%%%%%%%%%%%%%%%%%%%%%%%%%%%%%%%%%%%%%%%%%%%%%%%%%%%%%%%%%%%

\section{Non-asymptotic mean square error of the multilevel Monte Carlo estimator}\label{GF}

%%%%%%%%%%%%%%%%%%%%%%%%%%%%%%%%%%%%%%%%%%%%%%%%%%%%%%%%%%%
\subsection{Assumptions and strong error analysis}
It is well known that under assumption $(\mathcal H_{GL})$ we have 
$$
\forall p\geq 1, \sup_{0 \leq t \leq T}|X_t|,\sup_{0 \leq t \leq T}|X^n_t|\in L^p\mbox{ and } \mathbb{E}\left[{\sup_{0 \leq t \leq T}}\arrowvert X_t -  {X}^n_t\arrowvert^{p} \right] \leq\frac{K_p(T)}{n^{p/2}},\mbox{ with }  K_p(T)<\infty. 
\leqno(\mathcal P)
$$
Moreover, since the diffusion coefficient is constant, the Euler scheme coincides with the Milstein scheme and if $b$ belongs to $\mathscr C^{2}(\RR^d,\RR^d)$ with bounded derivatives, then the strong error estimation improves to $\mathbb{E}\left[{\sup_{0 \leq t \leq T}}\arrowvert X_t -  {X}_t^n\arrowvert^{p} \right] \leq\frac{K_p(T)}{n^{p}},\mbox{ with }  K_p(T)<\infty$ (see for instance \cite{milstein}). In order to get a non-asymptotic control of the bias and the variance of the multilevel Monte Carlo estimator, we are now going to state an explicit bound for the terminal quadratic strong error $\mathbb{E}\left[\arrowvert X^{mn}_{T} -  {X}^n_{T}\arrowvert^2 \right]$ for $(n,m)\in {\mathbb N}^*\times \bar{\mathbb N}$ (with the convention $X^{mn}=X$ for $m=\infty$) under the following assumption. The constancy of the diffusion coefficient ensures that the bias can be estimated with the right order of convergence using this strong error analysis instead of the more complicated weak error analysis.
\subsection*{Assumption (R1)}
The function $b\in \mathscr C^{2}(\RR^d,\RR^d)$ and there exist finite constants  $[\dot b]_{\infty}\in (0,+\infty)$ and $a_{\Delta b}\in [0,+\infty)$ such that
\begin{align}
   &\forall x\in\RR^d,\;\quad\|\nabla b(x)\|\leq [\dot b]_{\infty},\notag\\
&\forall x\in\RR^d,\;|\Delta b(x)|\le 2a_{\Delta b}(1+|x-x_0|).\label{afflapb}
\end{align}
% In the two following results, we use the conventions $\frac{e^{[\dot b]_\infty t}-1}{[\dot b]_\infty}=t=\frac{1-e^{-[\dot b]_\infty t}}{[\dot b]_\infty}$ and $\frac{[\dot b]_\infty t-1+e^{-[\dot b]_\infty t}}{[\dot b]^2_\infty}=\frac{t^2}{2}$ when $[\dot b]_\infty=0$, $\bar{\mathbb N}={\mathbb N}^*\cup\{\infty\}$, $\eta_\infty(t)=t$  and $X^\infty=X$. 
\begin{prop}\label{properfort}
Assume  ${\mathbf (\mathbf R\mathbf 1)}$. Then, for all  $(n,m)\in {\mathbb N}^*\times \bar{\mathbb N}$,
\begin{align*}
   &\E\left[|X^{mn}_{T}-X^n_{T}|^2\right]\le K_{1,m}\frac{(m-1)T^2}{mn^2}\mbox{ and }\E\left[\max_{0\le k\le n}|X^{mn}_{\frac{kT}{n}}-X^n_{\frac{kT}{n}}|^2\right]\le K_{2,m}\frac{(m-1)T^2}{mn^2}\\
&\mbox{where }\sqrt{K_{1,m}}=C_{\eqref{cte:maj}}[\dot b]_\infty  \sqrt{\frac{d(2m-1)}{6m}}+e^{[\dot b]_{\infty}T}\sqrt{\frac{m-1}{m}}\\&\times \bigg(\frac{|b(x_0)|}{2}\bigg(a_{\Delta b}\frac{[\dot b]_\infty T-1+e^{-[\dot b]_\infty T}}{2[\dot b]^2_\infty}+[\dot b]_\infty T\bigg)+a_{\Delta b}\frac{1-e^{-[\dot b]_\infty T}}{[\dot b]_\infty}+\frac{2}{{3}}\sqrt{d}([\dot b]_{\infty}^2+ a_{\Delta b})T^{3/2}\bigg)
\end{align*}
and $K_{2,m}$ is defined like $K_{1,m}$ but with $C_{\eqref{cte:maj}}+\sqrt{T}$ replacing $C_{\eqref{cte:maj}}=\sqrt{T}+[\dot b]_\infty\int_0^Te^{[\dot b]_\infty(T-t)}\sqrt{t} dt$.
\end{prop}
In the estimations (and in the remaining of the paper), $m$ only appears through ratios which have a limit as $m\to\infty$ and when $m=\infty$, we consider that they are equal to this limit. The proof is postponed to Section \ref{secu}.
\subsection{MLMC parameters optimization revisited}\label{sec:opti}
In what follows let us assume that $f\in \mathscr C^1(\RR^d,\RR)$ is  a Lipschitz continuous function with constant $[\dot f]_{\infty}$. For the Multilevel Monte Carlo estimator
\begin{equation*}
\displaystyle\ \hat Q=\frac{1}{N_0}\sum_{k=0}^{N_0}f( X^1_{T,k})+\sum_{\ell=1}^{L}\frac{1}{N_\ell}
\sum_{k=1}^{N_\ell}\left(f( X^{m^{\ell}}_{T,k})-f( X^{m^{\ell-1}}_{T,k})\right),
\end{equation*}
defined in \eqref{estimult}, the expectation leads to a telescoping summation so that 
\begin{equation}\label{upperbound:bias}
\left|\E\left[f(X_T)-\hat{Q}\right]\right|=\left|\E\left[f(X_T)-f(X^{m^L}_T)\right]\right|\le [\dot f]_\infty\E^{1/2}\left[\left|X_T-X^{m^L}_T\right|^2\right]\le [\dot f]_\infty \frac{T\sqrt{K_{1,\infty}}}{m^L},
\end{equation}
where we used Proposition \ref{properfort} for the last inequality. On the other hand, again by Proposition \ref{properfort},
$${\rm Var}\left[f( X^{m^{\ell}}_{T})-f( X^{m^{\ell-1}}_{T})\right]\le [\dot f]_\infty^2\EE\left[\left|X^{m^{\ell}}_{T}-X^{m^{\ell-1}}_{T}\right|^2\right]\le [\dot f]_\infty^2\frac{K_{1,m}(m-1)T^2}{m^{2\ell -1}}.$$
Last, as  $X^1_T\sim {\mathcal N}(x_0+b(x_0)T,TI_d)$, we use the logarithmic Sobolev inequality for the Gaussian measure and  the Herbst's argument (see e.g. propositions 5.5.1 and 5.4.1 in \cite{BakGenLed}) to get for all $\lambda\in \RR$
\begin{equation}\label{Co:gau}
\EE \bigg[\exp\bigg(\lambda (f(X^1_T)- \EE[f(X^1_T)])\bigg)\bigg]\leq \exp\left(\frac{\lambda^2[\dot f]^2_\infty T }{2}\right).
\end{equation}
% Now, thanks to  the Taylor expansion there exists $\theta_1\in(0,1)$ such that
% \begin{multline}\label{DL:1}
% \left|\EE \bigg[\exp\bigg(\lambda (f(X^1_T)- \EE[f(X^1_T)])\bigg)\bigg]-1-\frac{\lambda^2}{2}{\rm Var}\left[f(X^1_T)\right]  \right|\\\leq \frac{\lambda^3}{6}\mathbb E\left[\left|f(X^1_T)- \EE[f(X^1_T)]\right|^3\exp\bigg(\theta_1\lambda (f(X^1_T)- \EE[f(X^1_T)])\bigg)\right]
% \end{multline}
% and  $\theta_2\in(0,1)$
% \begin{equation}\label{DL:2}
% \left|\exp\left(\frac{\lambda^2[\dot f]^2_\infty T }{2}\right)-1-\frac{\lambda^2[\dot f]^2_\infty T }{2} \right|\leq 
% \frac{\lambda^4}{8}[\dot f]^4_\infty T^2\exp\left(\frac{\theta_2\lambda^2[\dot f]^2_\infty T }{2}\right).
% \end{equation}
% Hence, thanks to the triangle inequality, we get for all $\lambda\in \RR$
% \begin{align*}
%   {\rm Var}\left[f(X^1_T)\right]&\le \frac{2}{\lambda^2}\Bigg\{ \bigg|\EE \bigg[\exp\bigg(\lambda (f(X^1_T)- \EE[f(X^1_T)])\bigg)\bigg]-1-\frac{\lambda^2}{2}{\rm Var}[f(X^1_T)]  \bigg|\\&+ \bigg|\EE \bigg[\exp\bigg(\lambda (f(X^1_T)- \EE[f(X^1_T)])\bigg)\bigg]-1\bigg| \Bigg\}\\
%   & \le[\dot f]^2_\infty T+\frac{2}{\lambda^2}\Bigg\{ \bigg|\EE \bigg[\exp\bigg(\lambda (f(X^1_T)- \EE[f(X^1_T)])\bigg)\bigg]-1-\frac{\lambda^2}{2}{\rm Var}[f(X^1_T)]  \bigg|\\&+ \bigg| \exp\left(\frac{\lambda^2[\dot f]^2_\infty T }{2}\right)-1-\frac{\lambda^2[\dot f]^2_\infty T }{2}\bigg|\Bigg\}
% \end{align*}
% where  the last inequality is obtained by combining  \eqref{Co:gau} and once again the triangle inequality. Thus, using expansions \eqref{DL:1}, \eqref{DL:2} and  taking $\lambda=0$,  we conclude that
By performing Taylor expansions as $\lambda\to 0$, we easily deduce that
$$
{\rm Var}\left[f(X^1_T)\right]\le [\dot f]^2_\infty T.
$$
as a consequence, the following non-asymptotic estimation of the mean square error of $\hat{Q}$ holds.
\begin{prop}\label{MSE} Under ${\mathbf (\mathbf R\mathbf 1)}$,
$$
   \E\left[\left(\hat{Q}-\E[f(X_T)]\right)^2\right]\le [\dot f]^2_\infty\bigg(\frac{K_{1,\infty}T^2}{m^{2L}}+\frac{T}{N_0}+\sum_{\ell =1}^L\frac{K_{1,m}(m-1)T^2}{N_\ell m^{2\ell -1}}\bigg).
$$
\end{prop}
% \begin{remark}
% \textcolor{magenta}{Pour $f:(\R^d)^n\to\R$ Lipschitz pour la norme $\infty$ en $n$ des normes euclidiennes sur $\R^d$, $\hat Q=\frac{1}{N_0}\sum_{k=0}^{N_0}f( X^n_{t_1,\hdots,t_n,k})+\sum_{\ell=0}^{L}\frac{1}{N_\ell}
% \sum_{k=1}^{N_\ell}\left(f( X^{nm^{\ell}}_{t_1,\hdots,t_n,k})-f( X^{nm^{\ell-1}}_{t_1,\hdots,t_n,k})\right)$,
% \begin{align*}
%    \E\left[\left(\hat{Q}-\E[f(X_{t_1},\hdots,X_{t_n})]\right)^2\right]\le [\dot f]^2_\infty\bigg(&\frac{K^2_{2,\infty}T^2}{n^2m^{2L}}+\sum_{\ell =1}^L\frac{K_{2,m}(m-1)T^2}{N_\ell n^2m^{2\ell -1}}\\&+\frac{1}{N_0}\bigg(2e^{[\dot b]_\infty T}\sqrt{dT}+\frac{e^{[\dot b]_\infty T}-1}{[\dot b]_\infty}|b(x_0)|\bigg)^2\bigg).
% \end{align*}}
% \end{remark}

According to the above proposition, to achieve a root mean square error $\varepsilon>0$, one should choose 
\begin{equation}\frac{[\dot f]_\infty T\sqrt{K_{1,\infty}}}{m^{L}}<\varepsilon\mbox{ i.e. }L\ge \lfloor \ln([\dot f]_\infty T\sqrt{K_{1,\infty}}/\varepsilon)/\ln(m)\rfloor +1.\label{minol}
\end{equation} For such a choice, one should then choose $(N_\ell)_{0\le \ell\le L}$ such that \begin{equation}
   \sum_{\ell =1}^L\frac{m+1}{N_\ell m^{2\ell+1}}+\frac{C_{\eqref{varconst}}^2}{N_0}\le \frac{m+1}{K_{1,m}m^2(m-1)}\left(\frac{\varepsilon^2}{[\dot f]^2_\infty T^2}-\frac{K_{1,\infty}}{m^{2L}}\right)\label{varconst}
\end{equation} where $C_{\eqref{varconst}}=\frac{1}{m}\sqrt{\frac{m+1}{K_{1,m}(m-1)T}}$
minimizing the computation cost which is equal to $N_0+\sum_{\ell=1}^LN_\ell(m+1)m^{\ell-1}$. Note that for $\ell\in\{1,\hdots,L\}$, $(m+1)m^{\ell-1}=m^{\ell}+m^{\ell-1}$ is the number of grid values of the Euler schemes which are computed for each Brownian path at the level $\ell$. This constrained minimization problem leads to $N_0=N\frac{C_{\eqref{varconst}}}{C_{\eqref{varconst}}+\sum_{\ell =1}^Lm^{-3\ell/2}}$ and  $N_\ell=N\frac{m^{-3\ell/2}}{C_{\eqref{varconst}}+\sum_{\ell =1}^Lm^{-3\ell/2}}$ where the total number $N$ of simulations is chosen in order to achieve equality in \eqref{varconst} : 
\begin{equation}\label{eq:Nsize}
N=\left(C_{\eqref{varconst}}+\sum_{\ell =1}^Lm^{-3\ell/2}\right)\left(C_{\eqref{varconst}}+\frac{m+1}{m}\times\frac{1-m^{-L/2}}{\sqrt{m}-1}\right)\frac{K_{1,m}m^2(m-1)[\dot f]^2_\infty T^2}{(m+1)(\varepsilon^2-[\dot f]^2_\infty T^2K_{1,\infty}m^{-2L})}.
\end{equation}
Then the computation cost is given by ${\rm Cost}(m,m^{-L})$ where 
$${\rm Cost}(m,x)=\left(C_{\eqref{varconst}}+\frac{m+1}{m}\times\frac{1-\sqrt{x}}{\sqrt{m}-1}\right)^2\frac{K_{1,m}m^2(m-1)[\dot f]^2_\infty T^2}{(m+1)(\varepsilon^2-[\dot f]^2_\infty T^2K_{1,\infty}x^2)}.$$
Notice that for fixed $m$, ${\rm Cost}(m,x)$ is up to some positive multiplicative factor not depending on $x$ equal to $g(x)=\frac{(\sqrt{\alpha}-\sqrt{x})^2}{\beta^2_{\varepsilon}-x^2}$ with $\sqrt{\alpha}=1+\frac{m(\sqrt{m}-1)}{m+1}C_{\eqref{varconst}}>1$ and $\beta_{\varepsilon}=\frac{\varepsilon}{[\dot f]_\infty T\sqrt{K_{1,\infty}}}$ not depending on $x$. We thus want to find $L\in {\mathbb N}$ minimizing $g(m^{-L})$ under the constraint \eqref{minol} which writes $m^{-L}<\beta_{\varepsilon}$. We have
$g'(x)=\frac{\sqrt{\alpha}-\sqrt{x}}{(\beta^2_{\varepsilon}-x^2)^2}\left(2\sqrt{\alpha} x-x^{3/2}-\frac{\beta^2_{\varepsilon}}{\sqrt{x}}\right)$. Since, as $\alpha>1$, $x\mapsto h(x)= 2\sqrt{\alpha} x-x^{3/2}-\frac{\beta^2_{\varepsilon}}{\sqrt{x}}$ is increasing on $(0,1]$, \begin{itemize}
   \item either $2\sqrt{\alpha}-1-\beta^2_{\varepsilon}\le 0$, which implies $\beta_{\varepsilon}>1$ and $\inf_{x\in [0,1]}g(x)=g(1)$ so that $L=0$ solves the constrained minimization problem.
\item or $2\sqrt{\alpha}-1-\beta^2_{\varepsilon}>0$ so that, since $\lim_{x\to 0^+}2\sqrt{\alpha} x-x^{3/2}-\frac{\beta^2_{\varepsilon}}{\sqrt{x}}=-\infty$ and $2\sqrt{\alpha} \beta_{\varepsilon}-\beta^{3/2}_{\varepsilon}-\frac{\beta^2_{\varepsilon}}{\sqrt{\beta_{\varepsilon}}}=2\beta_{\varepsilon}(\sqrt{\alpha}-\sqrt{\beta_{\varepsilon}})>2\beta_{\varepsilon}(1-\sqrt{\beta_{\varepsilon}})$, there exists $x^\star_{\varepsilon}\in (0,1\wedge \beta_{\varepsilon})$ such that $g$ is decreasing on $(0,x^\star_{\varepsilon})$ and increasing on $(x^\star_{\varepsilon},1\wedge \beta_{\varepsilon})$ and the value of $L$ solving the constrained minimization problem belongs to $\{\lfloor -\frac{\ln x^\star_{\varepsilon}}{\ln m}\rfloor,\lceil -\frac{\ln x^\star_{\varepsilon}}{\ln m}\rceil\}$.
\end{itemize}
We denote by $L^\varepsilon$ this optimal value of $L$, by $N^\varepsilon$ (resp. $N^\varepsilon_\ell$) the corresponding total number of samples (resp. number of samples in the level $\ell$) and by $\hat{Q}_\varepsilon$ the multilevel Monte Carlo estimator \eqref{estimult} with those optimal parameters. 
When $\varepsilon^2<[\dot f]^2_\infty T^2K_{1,\infty}\left(1+\frac{2m(\sqrt{m}-1)C_{\eqref{varconst}}}{m+1}\right)$ which is equivalent to $2\sqrt{\alpha}-1-\beta^2_{\varepsilon}>0$, since $h\left(\frac{\beta^{4/3}_{\varepsilon}}{2^{2/3}\alpha^{1/3}}\right)=-\frac{\beta^{2}_{\varepsilon}}{2\sqrt{\alpha}}<0$ and $\inf_{x\in (0,1]}h'(x)\ge 2\sqrt{\alpha}-\frac{3}{2}>0$, we have $\frac{\beta^{4/3}_{\varepsilon}}{2^{2/3}\alpha^{1/3}}<x^\star_{\varepsilon}<\frac{\beta^{4/3}_{\varepsilon}}{2^{2/3}\alpha^{1/3}}+\frac{1}{2\sqrt{\alpha}-3/2}\times\frac{\beta^{2}_{\varepsilon}}{2\sqrt{\alpha}}$. Hence, as  $\varepsilon\to 0$, $x^\star_{\varepsilon}\sim \frac{\beta^{4/3}_{\varepsilon}}{2^{2/3}\alpha^{1/3}}$,  $L^\varepsilon\sim \frac{4\ln(1/\varepsilon)}{3\ln m}$ and the bias term $[\dot f]_\infty T\sqrt{K_{1,\infty}}m^{-L^{\varepsilon}}$ behaves as ${\mathcal O}(\varepsilon^{4/3})$. More precisely, when $\varepsilon<c_1:=[\dot f]_\infty T\sqrt{K_{1,\infty}}\left(1+\frac{2m(\sqrt{m}-1)C_{\eqref{varconst}}}{m+1}\right)^{1/2}$,
 \begin{align}&m^{-L^\varepsilon}\ge e^{-\lceil -\frac{\ln x^\star_{\varepsilon}}{\ln m}\rceil\ln m}>\frac{x^\star_{\varepsilon}}{m}>\frac{\varepsilon^{4/3}}{2^{2/3}m([\dot f]_\infty T\sqrt{K_{1,\infty}})^{4/3}(1+\frac{m(\sqrt{m}-1)}{m+1}C_{\eqref{varconst}})^{2/3}},\label{minobiais}\\
&m^{-L^\varepsilon}\le e^{-\lfloor -\frac{\ln x^\star_{\varepsilon}}{\ln m}\rfloor\ln m}<mx^\star_\varepsilon\notag\\&<\frac{m\varepsilon^{4/3}}{(1+\frac{m(\sqrt{m}-1)}{m+1}C_{\eqref{varconst}})([\dot f]_\infty T\sqrt{K_{1,\infty}})^{4/3}}\left(\frac{(1+\frac{m(\sqrt{m}-1)}{m+1}C_{\eqref{varconst}})^{1/3}}{2^{2/3}}+\frac{(1+\frac{2m(\sqrt{m}-1)}{m+1}C_{\eqref{varconst}})^{1/3}}{1+\frac{4m(\sqrt{m}-1)}{m+1}C_{\eqref{varconst}}}\right)\label{majobiais}\end{align}
We easily deduce that, as expected from \cite{Gil}, ${\rm Cost}(m,m^{-L^\varepsilon})={\mathcal O}(\varepsilon^{-2})$ as $\varepsilon\to 0$ for fixed $m$ and $N^\varepsilon={\mathcal O}(\varepsilon^{-2})$. One could also consider minimizing $m\mapsto {\rm Cost}(m,L^\varepsilon(m))$ numerically, where we used the notation $L^\varepsilon(m)$ to make the dependence on $m$ explicit.
\begin{remark}
For this optimal choice which clearly differs from the one in \cite{Gil}, the bias of the Multilevel Monte Carlo method is not of the same order of magnitude as the precision $\varepsilon$ but much smaller. To the best of our knowledge, such a $4/3$ order of convergence of the bias does not appear in the existing multilevel Monte Carlo methods literature. Notice that, for stochastic differential equations with a non constant diffusion coefficient (multiplicative noise), this property remains true for the multilevel Monte Carlo estimator based on the Giles and Szpruch scheme \cite{GilSzp}, since it exhibits the same orders of convergence of the bias and the variance within a given level as \eqref{estimult}.
\end{remark}

\section{Concentration bounds for the Multilevel Monte Carlo Euler method}\label{so}
The main result of this paper is a concentration inequality for the Multilevel Monte Carlo estimator $\hat{Q}$ defined in \eqref{estimult}. To prove this result, we are going to estimate the moment generating function of $$\hat{Q}-\EE[f( X^{m^{L}}_{T})]=\sum_{\ell=0}^L\hat{Q}_{\ell},$$
where $\hat Q_{0}:= \frac{1}{N_0}\sum_{k=1}^{N_0}f( X^1_{T,k})-\mathbb Ef(X^1_T) $ and, for $\ell\ge 1$, 
$$\hat Q_{\ell}:=\frac{1}{N_\ell}
\sum_{k=1}^{N_\ell}\left(f( X^{m^{\ell}}_{T,k})-f( X^{m^{\ell-1}}_{T,k})-\EE[f( X^{m^{\ell}}_{T})-
f( X^{m^{\ell-1}}_{T})]\right).$$
\subsection{Estimation of the moment generating function}\label{subsec:momgen}
We are first going to derive an exponential type upper bound with the optimal rate of convergence for  the moment generating function of the square of the error $U_T=X^n_T-X^{mn}_T$ between the Euler schemes with $n$ and $mn$ steps. The proof of the following result is postponed to Section \ref{secu}.
\begin{theorem}\label{laplace1}
Let $(n,m)\in {\mathbb N}^*\times\bar{\mathbb N}$, $t_k=\frac{kT}{n}$ for $k\in\{0,\hdots,n\}$  and $\rho$  be a constant satisfying 
\begin{equation}\label{rho:lap1}
0\leq  \rho \le \frac{9mn^2}{4T^2(m-1)\left(C_\eqref{cte:maj} [\dot b]_\infty\sqrt{3d(2m-1)/m}+C_\eqref{maju2}T^{3/2}\sqrt{2(m-1)/m}\right)^2}:=
\rho_{\eqref{rho:lap1}}n^2
\end{equation}
where $C_{\eqref{cte:maj}}=\sqrt{T}+[\dot b]_\infty\int_0^Te^{[\dot b]_\infty(T-t)}\sqrt{t} dt$ with $C_{\eqref{maju2}}=e^{T[\dot b]_\infty}([\dot b]^2_\infty+a_{\Delta b})$ and by convention $\frac{m}{m-1}=\frac{m-1}{m}=1$ and $\frac{2m-1}{m}=2$ when $m=\infty$.
Under assumption ${\mathbf (\mathbf R\mathbf 1)}$, we have for all $x\ge 0$\begin{align}\label{majufin}
   &\EE\left[  \exp\left\{ \rho \bigg(\frac{(m-1)T^2x}{2mn}+\max_{1\le k\le n}|X^{mn}_{t_k}-X^n_{t_k}|\bigg)^2 \right\}\right]\le \exp\left\{\rho C_{\eqref{majufin}}(x)\frac{(m-1)T^2}{mn^2}\right\}\mbox{ with }\\
&C_{\eqref{majufin}}(x)=\left(C_\eqref{cte:maj} [\dot b]_\infty\sqrt{\frac{3d(2m-1)}{m}}+C_\eqref{maju2}T^{3/2}\sqrt{\frac{2(m-1)}{m}}\right)\bigg(C_\eqref{cte:maj} [\dot b]_\infty\frac{2\ln 2}{3\sqrt{3}}\sqrt{\frac{d(2m-1)}{m}}\notag\\&+C_\eqref{maju2}\sqrt{\frac{(m-1)T}{2m}}\bigg(\frac{(3d+1)(|b(x_0)|+[\dot b]_\infty+\tilde x)^2}{4[\dot b]^2_\infty}+\frac{4d\sqrt{T}(|b(x_0)|+[\dot b]_\infty+\tilde x)}{3\sqrt{\pi}[\dot b]_\infty}+\frac{4dT\ln 2}{9}\bigg)\bigg),\notag
\end{align}
and $\tilde x=\frac{[\dot b]_\infty e^{-[\dot b]_\infty T}}{([\dot b]^2_\infty+a_{\Delta b})}x$.
\end{theorem}
To derive our main result, we need to reinforce our assumption ${\mathbf (\mathbf R\mathbf 1)}$ since our approach relies on Malliavin calculus that requires additional smoothness on the coefficient $b$.

\subsection*{Assumption (R2)}  The function $b\in \mathscr C^{3}(\RR^d,\RR^d)$ and satisfies assumption  ${\mathbf (\mathbf R\mathbf 1)}$. Moreover, there exist finite constants  $[\ddot b]_{\infty}\in (0,+\infty)$ and $a_{\nabla\Delta b}\in [0,+\infty)$ such that
\begin{align}
&\forall j\in\{1,\hdots,d\},\;\forall x\in \RR^d,  \;\left\|\frac{\partial\nabla b }{\partial x_j} (x)\right\|\leq [\ddot b]_\infty\notag\\
&\forall x\in \RR^d,  \;\left\|\nabla \Delta b (x)\right\|\leq 2a_{\nabla\Delta b}(1+|x-x_0|).\label{affgradlapb}
\end{align}
To state our next results we introduce the following finite quantities.
\subsection*{Constants Notations (CN)}  
\begin{align*}
 \rho_{\eqref{DUB:3}}&:=  \frac{m^2}{2C^2_{\eqref{majDu3}} T^2(m-1)^2},
\quad\rho_{\eqref{DUB:2}}:=\frac{3n^2}{4 T^2 d[\ddot b]^2_\infty (2m-1)(m-1) } ,\\ C_{\eqref{majDu3}}&:=
(\sqrt{d} [\dot b]_\infty[\ddot b]_\infty\vee[\dot b]^2_\infty +a_{\nabla\Delta b} 
),\quad\Phi_1(r):=\frac{\sqrt{d}[\ddot b]_\infty }{[\dot b]_\infty }(e^{[\dot b]_\infty (T- r)} -1),\\
\Phi_2(r)&:=\int_{r}^Te^{[\dot b]_\infty s}\sqrt{s}ds\,\mbox{ and } \,\Phi_3(r):=\sqrt{\frac{1-e^{-2[\dot b]_\infty(T-r)}}{2[\dot b]_\infty}}+\sqrt{\frac{[\dot b]_\infty}{2}}\int_r^T\sqrt{1-e^{-2[\dot b]_\infty(t-r)}}dt,\\
\hat\rho(r)&:=\frac{\frac{{\rho_{\eqref{rho:lap1}}}}{\Phi^2_1(r)}\frac{{\rho_{\eqref{DUB:3}}}}{\Phi^2_2(r)}\frac{{\rho_{\eqref{DUB:2}}}}{\Phi^2_3(r)}}
{\left(\frac{\sqrt{\rho_{\eqref{rho:lap1}}}}{\Phi_1(r)}{\frac{\sqrt{\rho_{\eqref{DUB:3}}}}{\Phi_2(r)}+\frac{\sqrt{\rho_{\eqref{rho:lap1}}}}{\Phi_1(r)}\frac{\sqrt{\rho_{\eqref{DUB:2}}}}{\Phi_3(r)}}+ {\frac{\sqrt{\rho_{\eqref{DUB:3}}}}{\Phi_2(r)}\frac{\sqrt{\rho_{\eqref{DUB:2}}}}{\Phi_3(r)}}\right)^2}, \\
\Phi(r,x)&:= \bigg({ \frac{\sqrt{\rho_{\eqref{rho:lap1}}}}{\Phi_1(r)}{\frac{\sqrt{\rho_{\eqref{DUB:3}}}}{\Phi_2(r)}+\frac{\sqrt{\rho_{\eqref{rho:lap1}}}}{\Phi_1(r)}\frac{\sqrt{\rho_{\eqref{DUB:2}}}}{\Phi_3(r)}}+ {\frac{\sqrt{\rho_{\eqref{DUB:3}}}}{\Phi_2(r)}\frac{\sqrt{\rho_{\eqref{DUB:2}}}}{\Phi_3(r)}}}\bigg)\\&\times\bigg(\frac{\Phi^2_3(r) C_\eqref{DUB:2}}{\frac{\sqrt{\rho_{\eqref{rho:lap1}}}}{\Phi_1(r)} \frac{\sqrt{\rho_{\eqref{DUB:3}}}}{\Phi_2(r)}}+\frac{\phi_2(r,x)}{\frac{\sqrt{\rho_{\eqref{rho:lap1}}}}{\Phi_1(r)} \frac{\sqrt{\rho_{\eqref{DUB:2}}}}{\Phi_3(r)}}+
\frac{\Phi^2_1(r)C_{\eqref{majufin}}(0)}{{\frac{\sqrt{\rho_{\eqref{DUB:3}}}}{\Phi_2(r)}\frac{\sqrt{\rho_{\eqref{DUB:2}}}}{\Phi_3(r)}}}\bigg), \;\; x\ge0,
\\
\phi_2(r,x)&:=\frac{(m-1)}{m}\bigg( (3d+1)\left(C_{\eqref{majDu3}}\frac{|b(x_0)|+[\dot b]_\infty}{2[\dot b]^2_\infty}(e^{[\dot b]_\infty T}-e^{[\dot b]_\infty r})+x\right)^2\\&+
\frac{4dC_{\eqref{majDu3}}\Phi_2(r)}{\sqrt{\pi}}\left(C_{\eqref{majDu3}}\frac{|b(x_0)|+[\dot b]_\infty}{2[\dot b]^2_\infty}(e^{[\dot b]_\infty T}-e^{[\dot b]_\infty r})+x\right)+d\ln 2C^2_{\eqref{majDu3}}\Phi^2_2(r)\bigg), x\ge0,\\
 C_{\eqref{majofin}}:=&d T\int_0^T\frac{e^{2[\dot b]_{\infty}(T-t) }}{\hat\rho(t)}\left(\sqrt{\rho_{\eqref{rho:lap1}}}+\frac{[\dot f]_{_{\rm{lip}}}}{[\dot f]_{\infty}}\sqrt{\hat\rho(t)}\right)^2dt\\\times&\sup_{r\in [0,T)}\frac{[\dot f]_{_{\rm{lip}}}\sqrt{\hat\rho(r)}C_{\eqref{majufin}}({2[\dot b]_{\infty}[\dot f]_{\infty}}/{T[\dot f]_{\rm{lip}}})+[\dot f]_{\infty}\hat\rho(r)\Phi(r,0)/\sqrt{\rho_{\eqref{rho:lap1}}}}{[\dot f]_{\infty}\sqrt{\rho_{\eqref{rho:lap1}}}+[\dot f]_{_{\rm{lip}}}\sqrt{\hat\rho(r)}}.
\end{align*}
The proof of the following theorem is postponed to Section \ref{sec:6}.
\begin{theorem}\label{thm:Orlicz}
Let assumption ${\mathbf (\mathbf R\mathbf 2)}$ hold and   
 $f\in \mathscr C^1(\RR^d,\RR)$ be a Lipschitz continuous function with constant $[\dot f]_{\infty}$ and such that $\nabla f$ is also Lipchitz with constant $[\dot f]_{\rm lip}$.
For all $\lambda\le {\mathcal C}\min_{1\le\ell\le L}N_\ell m^{\ell}$, where $${\mathcal C}=\left(\frac{\rho_{\eqref{rho:lap1}}}{2dm^2T\int_0^T\frac{e^{2[\dot b]_{\infty}(T-t) }}{\hat\rho(t)}\left([\dot f]_{\infty}\sqrt{\rho_{\eqref{rho:lap1}}}+[\dot f]_{_{\rm{lip}}}\sqrt{\hat\rho(t)}\right)^2dt}\right)^{1/2},$$
we have
\begin{align*}
   \EE\left[\exp\left(\lambda[\hat Q-\mathbb Ef(X^{m^L}_T)]\right)\right]\le \exp\left\{ \lambda^2[\dot f]^2_{\infty}\left( \frac{T}{2N_0}+\sum_{\ell=1}^L\frac{C_{\eqref{majofin}}(m-1)T^2}{N_\ell m^{2\ell-1}}\right)\right\}.
\end{align*}
\end{theorem}
According to Section \ref{GF}, the bias satisfies 
$$
|\mathbb E f(X^{m^L}_T)-\mathbb Ef(X_T)|\le \mathbb E^{1/2} |f(X^{m^L}_T)- f(X_T)|^2\le \frac{[\dot f]_{\infty}T\sqrt{K_{1,\infty}}}{m^L},
$$
so we easily deduce the following corollary.
\begin{corollary}\label{thm:mgf}
Under assumptions of Theorem \ref{thm:Orlicz}, we have, 
for all $|\lambda|\le {\mathcal C}\min_{1\le\ell\le L}N_\ell m^{\ell}$,
\begin{align*}
   \EE\left[\exp\left(\lambda[\hat Q-\mathbb Ef(X_T)]\right)\right]\le \exp\left\{ \lambda^2[\dot f]^2_{\infty}\left( \frac{T}{2N_0}+\sum_{\ell=1}^L\frac{C_{\eqref{majofin}}(m-1)T^2}{N_\ell m^{2\ell-1}}\right)+|\lambda| \frac{[\dot f]_{\infty}T\sqrt{K_{1,\infty}}}{m^L}\right\}.
\end{align*}
\end{corollary}
\subsection{Concentration bounds}\label{sec:Concentrationbounds}
Using the above corollary,   for all  $\lambda \in[0,{\mathcal C}\min_{1\le\ell\le L}m^\ell N_\ell]$ and $\alpha\ge0$,  we get
 \begin{equation}\label{CI}
\PP\left(\hat Q-\mathbb Ef(X_T)\geq \alpha\right)\leq \exp\left\{ \psi_\alpha(\lambda)\right\},
\end{equation}
with
$$
 \psi_\alpha(\lambda):=\lambda^2 [\dot f]^2_{\infty}\left( \frac{T}{2N_0}+\sum_{\ell=1}^L\frac{C_{\eqref{majofin}}(m-1)T^2}{N_\ell m^{2\ell-1}}\right)+\lambda \left(\frac{[\dot f]_{\infty}T\sqrt{K_{1,\infty}}}{m^L} -\alpha\right).
$$
Now, when $0\le\alpha\le 2[\dot f]^2_{\infty}\left( \frac{T}{2N_0}+\sum_{\ell=1}^L\frac{C_{\eqref{majofin}}(m-1)T^2}{N_\ell m^{2\ell-1}}\right) {\mathcal C}\min_{1\le\ell\le L}m^\ell N_\ell+\frac{[\dot f]_{\infty}T\sqrt{K_{1,\infty}}}{m^L}$,
$$\min_{\lambda\in[0,{\mathcal C}\min_{1\le\ell\le L}m^\ell N_\ell]}\psi_\alpha(\lambda)=-\frac{\left(\alpha-\frac{[\dot f]_{\infty}T\sqrt{K_{1,\infty}}}{m^L}\right)_+^2}{4[\dot f]^2_{\infty}\left( \frac{T}{2N_0}+\sum_{\ell=1}^L\frac{C_{\eqref{majofin}}(m-1)T^2}{N_\ell m^{2\ell-1}}\right) },\;\; \mbox{ where } (x)_+=\max(x,0).$$
and otherwise
\begin{align*}
   \min_{\lambda\in[0,{\mathcal C}\min_{1\le\ell\le L}m^\ell N_\ell]}\psi_\alpha(\lambda)&={\mathcal C}\min_{1\le\ell\le L}m^\ell N_\ell\\&\times\bigg(\frac{[\dot f]_{\infty}T\sqrt{K_{1,\infty}}}{m^L}-\alpha+ [\dot f]^2_{\infty}\bigg( \frac{T}{2N_0}+\sum_{\ell=1}^L\frac{C_{\eqref{majofin}}(m-1)T^2}{N_\ell m^{2\ell-1}} \bigg){\mathcal C}\min_{1\le\ell\le L}m^\ell N_\ell\bigg)\\&< -\frac{{\mathcal C}\min_{1\le\ell\le L}m^\ell N_\ell}{2}\alpha.
\end{align*}
Dealing with $\PP\left(\hat Q-\mathbb Ef(X_T)\le-\alpha\right)$ in a symmetric way we end up with the concentration inequality, 
\begin{align*}
   \forall 0\le \alpha \le&2[\dot f]^2_{\infty}\left( \frac{T}{2N_0}+\sum_{\ell=1}^L\frac{C_{\eqref{majofin}}(m-1)T^2}{N_\ell m^{2\ell-1}}\right) {\mathcal C}\min_{1\le\ell\le L}m^\ell N_\ell+\frac{[\dot f]_{\infty}T\sqrt{K_{1,\infty}}}{m^L}\Bigg],\\
&\PP\left(|\hat Q-\mathbb Ef(X_T)|\ge\alpha\right)\le 2\exp\left(-\frac{\left(\alpha-\frac{[\dot f]_{\infty}T\sqrt{K_{1,\infty}}}{m^L}\right)^2}{4 [\dot f]^2_{\infty}\left( \frac{T}{2N_0}+\sum_{\ell=1}^L\frac{C_{\eqref{majofin}}(m-1)T^2}{N_\ell m^{2\ell-1}}\right) }\right)\\
\forall \alpha\ge &2[\dot f]^2_{\infty}\left( \frac{T}{2N_0}+\sum_{\ell=1}^L\frac{C_{\eqref{majofin}}(m-1)T^2}{N_\ell m^{2\ell-1}}\right) {\mathcal C}\min_{1\le\ell\le L}m^\ell N_\ell+ \frac{[\dot f]_{\infty}T\sqrt{K_{1,\infty}}}{m^L},\\&\PP\left(|\hat Q-\mathbb Ef(X_T)|\ge\alpha\right)\le 2\exp\left(-\frac{{\mathcal C}\min_{1\le\ell\le L}m^\ell N_\ell}{2}\alpha\right).\notag
\end{align*}
Hence, we proved our main result, which we now state.  
\begin{theorem}\label{thm:CI} Under assumptions of Theorem \ref{thm:Orlicz}, the multilevel Monte Carlo  estimator \eqref{estimult} satisfies,
 $\forall\;0\le\alpha\le 2[\dot f]^2_{\infty}\left( \frac{T}{2N_0}+\sum_{\ell=1}^L\frac{C_{\eqref{majofin}}(m-1)T^2}{N_\ell m^{2\ell-1}}\right) {\mathcal C}\min_{1\le\ell\le L}m^\ell N_\ell+ \frac{[\dot f]_{\infty}T\sqrt{K_{1,\infty}}}{m^L}$, \begin{equation*}
\PP\left(|\hat Q-\mathbb Ef(X_T)|\ge\alpha\right)\le 2\exp\left(-\frac{\left(\alpha- \frac{[\dot f]_{\infty}T\sqrt{K_{1,\infty}}}{m^L}\right)_+^2 }{2 [\dot f]^2_{\infty}\left( \frac{T}{N_0}+\sum_{\ell=1}^{L}\frac{2C_{\eqref{majofin}}(m-1)T^2}{N_\ell m^{2\ell-1}}\right) }\right),
\end{equation*}
with $(x)_+=\max(x,0).$
\end{theorem}

Notice that the factor $[\dot f]^2_{\infty}\left( \frac{T}{N_0}+\sum_{\ell=1}^{L}\frac{2C_{\eqref{majofin}}(m-1)T^2}{N_\ell m^{2\ell-1}}\right)$ in the denominator is closely related to the non-asymptotic upper-bound $[\dot f]^2_\infty\bigg(\frac{T}{N_0}+\sum_{\ell =1}^L\frac{K_{1,m}(m-1)T^2}{N_\ell m^{2\ell -1}}\bigg)$ of the variance of $\hat{Q}$ derived in Section \ref{sec:opti}. The only difference is the replacement of $K_{1,m}$ by $2C_{\eqref{majofin}}$.
Let us now discuss the constraint on $\alpha$ under which we proved Gaussian type concentration and see that in the limit $\varepsilon\to 0$, for the optimal parameters discussed in Section \ref{sec:opti}, we can choose $\alpha={\mathcal O}(\varepsilon^{2/3})$ i.e. much larger than the root mean square error $\varepsilon$.

Following the discussion and notations of Section \ref{sec:opti}, for $\varepsilon>0$, we consider $\hat Q_\varepsilon$ the MLMC estimator \eqref{estimult} with the optimal parameters 
 $L^{\varepsilon}$, $N^\varepsilon$, 
 \begin{equation}\label{optimal:param}
 N^\varepsilon_0=N^\varepsilon\frac{C_{\eqref{varconst}}}{C_{\eqref{varconst}}+\sum_{\ell =1}^{L^{\varepsilon}}m^{-3\ell/2}}\mbox{ and }
 N^\varepsilon_\ell=N^\varepsilon\frac{m^{-3\ell/2}}{C_{\eqref{varconst}}+\sum_{\ell =1}^{L^{\varepsilon}}m^{-3\ell/2}}, \ell\ge1.
 \end{equation}
One has $\min_{1\le \ell\le L^\varepsilon}m^\ell N^\varepsilon_\ell=\frac{N^\varepsilon m^{-L^\varepsilon/2}}{C_{\eqref{varconst}}+\sum_{\ell =1}^{L^\varepsilon}m^{-3\ell/2}}$ and therefore
$\frac{1}{N^\varepsilon_0}\min_{1\le \ell\le L^\varepsilon}m^\ell N^\varepsilon_\ell=\frac{m^{-L^\varepsilon/2}}{C_{\eqref{varconst}}}$.
As a consequence,
$$\frac{[\dot f]_{\infty}T\sqrt{K_{1,\infty}}}{m^{L^\varepsilon}}+2[\dot f]^2_{\infty}\left( \frac{T}{2N^\varepsilon_0}+\sum_{\ell=1}^{L^\varepsilon}\frac{C_{\eqref{majofin}}(m-1)T^2}{N^\varepsilon_\ell m^{2\ell-1}}\right){\mathcal C}\min_{1\le\ell\le L^\varepsilon}m^\ell N^\varepsilon_\ell\ge \frac{{\mathcal C}[\dot f]^2_{\infty}T}{C_{\eqref{varconst}}}m^{-L^\varepsilon/2}$$
where, according to \eqref{minobiais}, when $\varepsilon<c_1:=[\dot f]_\infty T\sqrt{K_{1,\infty}}\left(1+\frac{2m(\sqrt{m}-1)C_{\eqref{varconst}}}{m+1}\right)^{1/2}$, the right-hand side is larger than
$$c_2\varepsilon^{2/3}\mbox{ with }c_2:=\frac{{\mathcal C}[\dot f]^2_{\infty}T}{2^{1/3}C_{\eqref{varconst}}\sqrt{m}([\dot f]_\infty T\sqrt{K_{1,\infty}})^{2/3}(1+\frac{m(\sqrt{m}-1)}{m+1}C_{\eqref{varconst}})^{1/3}}.
$$
Under the same condition on $\varepsilon$, according to \eqref{majobiais}, 
\begin{align*}
  & [\dot f]_{\infty}T\sqrt{K_{1,\infty}}m^{-L^\varepsilon}\\&<\frac{m\varepsilon^{4/3}}{(1+\frac{m(\sqrt{m}-1)}{m+1}C_{\eqref{varconst}})([\dot f]_\infty T\sqrt{K_{1,\infty}})^{1/3}}\left(\frac{(1+\frac{m(\sqrt{m}-1)}{m+1}C_{\eqref{varconst}})^{1/3}}{2^{2/3}}+\frac{(1+\frac{2m(\sqrt{m}-1)}{m+1}C_{\eqref{varconst}})^{1/3}}{1+\frac{4m(\sqrt{m}-1)}{m+1}C_{\eqref{varconst}}}\right)
\end{align*}
On the other hand, one has
\begin{align}\label{ineq:bv}
   2 [\dot f]^2_{\infty}&\left( \frac{T}{N^\varepsilon_0}+\sum_{\ell=1}^{L^\varepsilon}\frac{2C_{\eqref{majofin}}(m-1)T^2}{N^\varepsilon_\ell m^{2\ell-1}}\right)\\&\le \left(2\vee\frac{4C_{\eqref{majofin}}}{K_{1,m}}\right)[\dot f]^2_\infty\bigg(\frac{K_{1,\infty}T^2}{m^{2L^\varepsilon}}+\frac{T}{N^\varepsilon_0}+\sum_{\ell =1}^{L^\varepsilon}\frac{K_{1,m}(m-1)T^2}{N^\varepsilon_\ell m^{2\ell -1}}\bigg)\le\left(2\vee\frac{4C_{\eqref{majofin}}}{K_{1,m}}\right)\varepsilon^2,
\end{align}
according to the optimization of parameters which follows Proposition \ref{MSE}.
Combining the two last inequalities and the fact that for positive $\alpha$ and $x$, $(\alpha-x)_+^2\ge\frac{\alpha^2}{2}-x^2$, we obtain that for $\varepsilon<c_1$,
\begin{align*}
   \frac{\left(\alpha- \frac{[\dot f]_{\infty}T\sqrt{K_{1,\infty}}}{m^{L^\varepsilon}}\right)_+^2 }{2 [\dot f]^2_{\infty}\left( \frac{T}{N^\varepsilon_0}+\sum_{\ell=1}^{L^\varepsilon}\frac{2C_{\eqref{majofin}}(m-1)T^2}{N^\varepsilon_\ell m^{2\ell-1}}\right) }\ge \frac{\alpha^2}{c_3\varepsilon^2}-c_4\varepsilon^{2/3}
\end{align*}
with
\begin{align*}
  c_3&=\left(4\vee\frac{8C_{\eqref{majofin}}}{K_{1,m}}\right)\\
  c_4&=\frac{2}{c_3}\left(\frac{m}{(1+\frac{m(\sqrt{m}-1)}{m+1}C_{\eqref{varconst}})([\dot f]_\infty T\sqrt{K_{1,\infty}})^{1/3}}\left(\frac{(1+\frac{m(\sqrt{m}-1)}{m+1}C_{\eqref{varconst}})^{1/3}}{2^{2/3}}+\frac{(1+\frac{2m(\sqrt{m}-1)}{m+1}C_{\eqref{varconst}})^{1/3}}{1+\frac{4m(\sqrt{m}-1)}{m+1}C_{\eqref{varconst}}}\right)\right)^2.
\end{align*}
Then
$$\forall \varepsilon\in (0,c_1),\;\forall \alpha\in(0,c_2\varepsilon^{2/3}),\;\PP\left(|\hat Q_\varepsilon-\mathbb Ef(X_T)|\ge\alpha\right)\le 2e^{-\frac{\alpha^2}{c_3\varepsilon^2}+c_4\varepsilon^{2/3}}.$$
Notice that the fact $[\dot f]_{\infty}T\sqrt{K_{1,\infty}}m^{-L^\varepsilon}<\varepsilon$ that the bias is smaller than the precision $\varepsilon$ leads, by a similar reasoning, to the following result.
\begin{corollary} 
Under assumptions of Theorem \ref{thm:Orlicz}, 
 the multilevel Monte Carlo  estimator $\hat Q_\varepsilon$ \eqref{estimult} equipped 
 with the optimal parameters \eqref{optimal:param} satisfies
$$\forall\varepsilon\in (0,c_1),\; \forall\alpha\in(0,c_2\varepsilon^{2/3}),\;\PP\left(|\hat Q_\varepsilon-\mathbb Ef(X_T)|\ge\alpha\right)\le 2e^{-\frac{\alpha^2}{c_3\varepsilon^2}}e^{\frac{2}{c_3}\wedge c_4\varepsilon^{2/3}}.$$
\end{corollary}

At this stage, a natural question arises: is there  an alternative choice of the  parameters that does not increase neither the root mean square error $\varepsilon$ nor the order in $\varepsilon$ of the computational cost of the multilevel Monte Carlo  estimator and 
for which the upper bound on the deviation parameter $\alpha$  is larger than $c_1\varepsilon^{2/3}$ ?

For $\beta>1$, we set $N^{\varepsilon,\beta}_\ell=N^\varepsilon_\ell\times \frac{m^{\frac{\ell -1}{2}}}{m^{\frac{\ell -1}{2}}\wedge {\ell}^\beta}$ for $\ell \in \{1,\hdots,L^\varepsilon\}$ and 
\begin{equation}\label{super:MLMC}
\hat{Q}_{\varepsilon,\beta}=\frac{1}{N_0^{\varepsilon}}\sum_{k=0}^{N^\varepsilon_0}f( X^1_{T,k})+\sum_{\ell=1}^{L^\varepsilon}\frac{1}{N^{\varepsilon,\beta}_\ell}
\sum_{k=1}^{N^{\varepsilon,\beta}_\ell}\left(f( X^{m^{\ell}}_{T,k})-f( X^{m^{\ell-1}}_{T,k})\right).
\end{equation}
Since for each $\ell$, $N^{\varepsilon,\beta}_\ell\ge N^\varepsilon_\ell$, the root mean square error and the statistical error of $\hat{Q}_{\varepsilon,\beta}$ are not greater than the ones of $\hat{Q}_{\varepsilon}$ and therefore than $\varepsilon$. The two estimators share the same bias.
Moreover, $\min_{1\le \ell\le L^\varepsilon}m^\ell N^{\varepsilon,\beta}_\ell\ge \frac{N^\varepsilon m^{-1/2}}{(L^\varepsilon)^\beta(C_{\eqref{varconst}}+\sum_{\ell =1}^{L^\varepsilon}m^{-3\ell/2})}$ so that$\frac{1}{N^\varepsilon_0}\min_{1\le \ell\le L^\varepsilon}m^\ell N^{\varepsilon,\beta}_\ell\ge \frac{1}{C_{\eqref{varconst}}\sqrt{m}(L^\varepsilon)^\beta}$.
As a consequence,
$$\frac{[\dot f]_{\infty}T\sqrt{K_{1,\infty}}}{m^{L^\varepsilon}}+2[\dot f]^2_{\infty}\left( \frac{T}{2N^\varepsilon_0}+\sum_{\ell=1}^{L^\varepsilon}\frac{C_{\eqref{majofin}}(m-1)T^2}{N^{\varepsilon,\beta}_\ell m^{2\ell-1}}\right){\mathcal C}\min_{1\le\ell\le L^\varepsilon}m^\ell N^{\varepsilon,\beta}_\ell\ge \frac{{\mathcal C}[\dot f]^2_{\infty}T}{C_{\eqref{varconst}}\sqrt{m}(L^\varepsilon)^\beta}$$
where, when $\varepsilon<c_1$, the right-hand side is larger than
$$c_5(\varepsilon):=\frac{{\mathcal C}[\dot f]^2_{\infty}T}{C_{\eqref{varconst}}\sqrt{m}}\left(1+\frac{4}{3\ln m}\ln\bigg(\frac{[\dot f]_{\infty}T\sqrt{2K_{1,\infty}(1+\frac{m(\sqrt{m}-1)}{m+1}C_{\eqref{varconst}})}}{\varepsilon}\bigg)\right)^{-\beta}.
$$
Reasoning like in the above derivation of concentration inequalities for $\hat Q_\varepsilon$, we easily get the following result.
\begin{corollary}\label{cordev2}
Under assumptions of Theorem \ref{thm:Orlicz},  the multilevel Monte Carlo  estimator $\hat Q_{\varepsilon,\beta}$ \eqref{super:MLMC}  satisfies
 $$\forall\varepsilon\in (0,c_1),\;\forall\alpha\in(0,c_5(\varepsilon)),\;
\PP\left(|\hat Q_{\varepsilon,\beta}-\mathbb Ef(X_T)|\ge\alpha\right)\le 2e^{-\frac{\alpha^2}{c_3\varepsilon^2}}e^{\frac{2}{c_3}\wedge c_4\varepsilon^{2/3}}.
$$
\end{corollary}
Moreover, the computational cost of $\hat{Q}_{\varepsilon,\beta}$ is proportional to\begin{align*}
   N^{\varepsilon}_0+\frac{m+1}{m}\sum_{\ell=1}^{L^\varepsilon}m^\ell N^{\varepsilon,\beta}_\ell=\frac{N^\varepsilon}{C_{\eqref{varconst}}+\sum_{\ell =1}^{L^{\varepsilon}}m^{-3\ell/2}}
\left(C_{\eqref{varconst}}+\frac{m+1}{m\sqrt{m}}\sum_{\ell=1}^{L^\varepsilon}(m^{\frac{\ell -1}{2}}\wedge {\ell}^\beta)^{-1}\right).\end{align*} Since $\beta>1$, it is of the same order ${\mathcal O}(\varepsilon^{-2})$ as $N^\varepsilon$ and therefore as the computational cost of $\hat Q_\varepsilon$ in the limit $\varepsilon\to 0$.

\subsection{Error control in Orlicz norm }\label{sec:orlicz}
In view of our previous results, it is natural and fruitful to generalize the RMS  error analysis developed in Section \ref{sec:opti} by means of Young functions that are increasing convex functions $\Psi: \mathbb R^+\rightarrow\mathbb R^+$ satisfying $\Psi(0)=0$ and $\lim_{x\to+\infty}\Psi(x)=+\infty$. For a given  Young function $\Psi$ the associated Orlicz norm $\|X\|_\Psi$ of a random variable $X$ is defined by 
$$
\|X\|_{\Psi}:=\inf\{c>0, \; : \mathbb E[\Psi({X}/{c})]\leq 1\} \mbox{ with  } \inf \emptyset =\infty.
$$
We set $\Psi_e(x):=\left(e^x-1\right)/(e-1)$ as a fixed Young function. At first,  let us deal with a standard Monte Carlo algorithm that approximates  $\mathbb E[f(X_T)]$ by $\bar Q:=\frac{1}{N_0}\sum_{k=1}^{N_0}f( X^{m^L}_{T,k})$.
 Then, one can use Section 4 of \cite{FriMen} to bound $\EE\left[\exp(\lambda (\bar Q-\mathbb E[f(X^{m^L}_T)])\right]$ from above. More precisely, taking advantage of the constant diffusion coefficient to improve the bound in Proposition 4.1 of \cite{FriMen} to $[f_i^\Delta]_1\le [\dot f]_\infty e^{[\dot b]_\infty (T-t_i)}$, we get
\begin{equation}
 \label{CI1}
 \EE\left[\exp(\lambda (\bar Q-\mathbb E[f(X^{m^L}_T)])\right]\leq \exp\left\{ \frac{\lambda^2C_{\eqref{CI1}}}{N_0}\right\}\mbox{ with }C_{\eqref{CI1}}=\frac{T[\dot f]_\infty^2}{2m^L}\sum_{k=1}^{m^L}e^{2[\dot b]_\infty (T-\frac{Tk}{m^L})}.
\end{equation}
With \eqref{upperbound:bias}, we easily deduce that for all $c>0$
$$
\mathbb E\left[\Psi_e\left(\frac{ \bar Q-\mathbb E[f(X_T)]}{c}\right) \right]\le \frac{1}{e-1}\left[\exp\left(\frac{C_{\eqref{CI1}}}{c^2 N_0}+\frac{[\dot f]_{\infty}T\sqrt{K_{1,\infty}}}{cm^L}\right)-1\right]
$$
and then 
$$
\left\| \bar Q-\mathbb E[f(X_T)] \right\|_{\Psi_e}\leq 
\frac{[\dot f]_{\infty}T\sqrt{K_{1,\infty}}}{2m^L}+\sqrt{\left(\frac{[\dot f]_{\infty}T\sqrt{K_{1,\infty}}}{2m^L}\right)^2+\frac{C_{\eqref{CI1}}}{ N_0}}.
$$
Recall that the RMS error in this case is given by  
\begin{equation*}
\EE^{1/2}\left[\left(\bar Q-\mathbb E[f(X_T)]\right)^2 \right]\le\sqrt{\left(\frac{[\dot f]_{\infty}T\sqrt{K_{1,\infty}}}{m^L}\right)^2+{{\rm Var}(\bar Q_{0})}}.
\end{equation*}
By performing Taylor expansions as $\lambda\to 0$ in  \eqref{CI1}, we easily get
$
{\rm Var}\left[\bar Q\right]\le \frac{2C_{\eqref{CI1}}}{N_0}.
$
As a consequence, we get 
\begin{equation}\label{MSE:MC}
\EE^{1/2}\left[\left(\bar Q-\mathbb E[f(X_T)]\right)^2 \right]\le\sqrt{\left(\frac{[\dot f]_{\infty}T\sqrt{K_{1,\infty}}}{m^L}\right)^2+\frac{2C_{\eqref{CI1}}}{N_0}}.
\end{equation}
Note that using 
\begin{equation}\label{ineq:orlicz_rmse}
\sqrt{\frac{a^2+b}{2}}\le\frac{a}{2} +\sqrt{\left(\frac{a}{2}\right)^2+\frac{b}{2}}\le\sqrt{\frac{(3+\sqrt{3})}{4}(a^2+b)}, \mbox{ for }a,b>0,
\end{equation}
we easily get
\begin{equation}\label{Orlicz:MC}
\sqrt{\frac{\left(\frac{[\dot f]_{\infty}T\sqrt{K_{1,\infty}}}{m^L}\right)^2+\frac{2C_{\eqref{CI1}}}{N_0}}{2}}\le \left\| \bar Q-\mathbb E[f(X_T)] \right\|_{\Psi_e}\leq \sqrt{\frac{(3+\sqrt{3})}{4}\left[\left(\frac{[\dot f]_{\infty}T\sqrt{K_{1,\infty}}}{m^L}\right)^2+\frac{2C_{\eqref{CI1}}}{N_0}\right]}.
\end{equation}
Then, the RMS error and the Orlicz norm of the standard Monte Carlo share, up to a constant factor, the same order.
%Choosing the optimal parameters $\bar N^\varepsilon_0$ and $\bar L^\varepsilon$ that achieves an error $\varepsilon$ in the above upper bound we get
%$$\left\| \bar Q^{\varepsilon}_{0}-\mathbb E[f(X_T)] \right\|_{\Psi_e}\leq (\frac{1}{2}+\sqrt{2})\varepsilon,$$
%where $\bar Q^{\varepsilon}_{0}$ is the Monte Carlo estimator equipped with $\bar N^\varepsilon_0$ and $\bar L^\varepsilon$.

Now, we proceed similarly for the multilevel Monte Carlo algorithm and thanks to Corollary \ref{thm:mgf} we write,  for all $c\ge 1/ ({\mathcal C}\min_{1\le\ell\le L}N_\ell m^{\ell})$
\begin{multline*}
\mathbb E\left[\Psi_e\left(\frac{\hat Q-\mathbb E[f(X_T)]}{c}\right) \right]\le \frac{1}{e-1}\\\times\left[\exp\left(\frac{[\dot f]^2_{\infty}}{c^2}\left( \frac{T}{2N_0}+\sum_{\ell=1}^L\frac{C_{\eqref{majofin}}(m-1)T^2}{N_\ell m^{2\ell-1}}\right)+\frac{[\dot f]_{\infty}T\sqrt{K_{1,\infty}}}{cm^L}\right)-1\right].
\end{multline*}
Hence,  if 
\begin{multline}\label{cd:orlicz}
\left(\frac{[\dot f]_{\infty}T\sqrt{K_{1,\infty}}}{2m^L}+\sqrt{\left(\frac{[\dot f]_{\infty}T\sqrt{K_{1,\infty}}}{2m^L}\right)^2+[\dot f]^2_{\infty}\left( \frac{T}{2N_0}+\sum_{\ell=1}^L\frac{C_{\eqref{majofin}}(m-1)T^2}{N_\ell m^{2\ell-1}}\right)}\right)\ge \\\frac{1}{{\mathcal C}\min_{1\le\ell\le L}N_\ell m^{\ell}}, 
\end{multline}
then
\begin{multline*}
\left\|\hat Q-\mathbb E[f(X_T)] \right\|_{\Psi_e}\leq \\
\left(\frac{[\dot f]_{\infty}T\sqrt{K_{1,\infty}}}{2m^L}+\sqrt{\left(\frac{[\dot f]_{\infty}T\sqrt{K_{1,\infty}}}{2m^L}\right)^2+[\dot f]^2_{\infty}\left( \frac{T}{2N_0}+\sum_{\ell=1}^L\frac{C_{\eqref{majofin}}(m-1)T^2}{N_\ell m^{2\ell-1}}\right)}\right)
\end{multline*}
and thus by \eqref{ineq:orlicz_rmse}, we get
\begin{equation}\label{Orlicz:MLMC}
\left\|\hat Q-\mathbb E[f(X_T)] \right\|_{\Psi_e}\leq c_6\sqrt{\left(\frac{[\dot f]_{\infty}T\sqrt{K_{1,\infty}}}{m^L}\right)^2+[\dot f]^2_{\infty}\left( \frac{T}{N_0}+\sum_{\ell=1}^L\frac{K_{1,m}(m-1)T^2}{N_\ell m^{2\ell-1}}\right)}, 
\end{equation}
where $ c_6:=\left[\left(\frac{3+\sqrt{3}}{4}\right)\left(1 \vee \frac{2C_{\eqref{majofin}}}{K_{1,m}}\right)\right]^{1/2}$.

Note that the above upper bound is equal, up to a constant factor,  to the upper bound given by the RMS error estimate of  Proposition \ref{MSE}. Combining this result with \eqref{MSE:MC}, \eqref{Orlicz:MC} and \eqref{Orlicz:MLMC}, we conclude that compared to the standard Monte Carlo method, under constraint \eqref{cd:orlicz}, the multilevel Monte Carlo estimator achieves the  same complexity reduction for the  Orlicz norm $\|\cdot\|_{\Psi_e}$ as for the RMS error.

Now, it remains to check the validity of constraint \eqref{cd:orlicz} when choosing  the multilevel Monte Carlo algorithm optimal parameters derived in Section \ref{sec:opti}. Let us recall that in this setting the RMS error upper bound is equal to $\varepsilon$. 
On the one hand, using \eqref{ineq:orlicz_rmse}, we deduce that the term in the left hand side of  \eqref{cd:orlicz} is larger than $c'_6\varepsilon$, where $c'_6:=\left[\left(\frac{1}{2} \wedge \frac{C_{\eqref{majofin}}}{K_{1,m}}\right)\right]^{1/2}$.
%On the one hand, using \eqref{ineq:bv}, the above upper bound is less or equal than
%$c_6 \varepsilon :=(\frac{1}{2}+\sqrt{\left(\frac{1}{2}\vee\frac{C_{\eqref{majofin}}}{K_{1,m}}\right)})\varepsilon$. 
On the other hand, combining $\min_{1\le \ell\le L^\varepsilon}m^\ell N^\varepsilon_\ell=\frac{N^\varepsilon m^{-L^\varepsilon/2}}{C_{\eqref{varconst}}+\sum_{\ell =1}^{L^\varepsilon}m^{-3\ell/2}}$ together with \eqref{eq:Nsize} and \eqref {minobiais}, we get for all $0\le \varepsilon \le c_1$
$$
 {\mathcal C}\min_{1\le\ell\le L}N_\ell m^{\ell}\ge  {\mathcal C} C_{\eqref{varconst}}\frac{K_{1,m}m^2(m-1)[\dot f]^2_\infty T^2m^{-L^\varepsilon/2}}{(m+1)\varepsilon^2}
 \ge c_7 {\varepsilon^{-4/3}}
$$
where $c_7:={\mathcal C} C_{\eqref{varconst}}\frac{K_{1,m}m^{3/2}(m-1)([\dot f]_\infty T)^{4/3}}{2^{1/3}(m+1)(K_{1,\infty})^{1/3}(1+\frac{m(\sqrt{m}-1)}{m+1}C_{\eqref{varconst}})^{1/3}}$. 
Hence, for all $0\le \varepsilon\le \left(c_6'c_7\right)^{3}\wedge c_1$  condition  \eqref{cd:orlicz} is satisfied and
$$
\left\|{\hat Q}_\varepsilon-\mathbb E[f(X_T)] \right\|_{\Psi_e}\leq c_6\varepsilon.$$

%%%%%%%%%%%%%%%%%%%%%%%%%%%%%%%%%%%%%%%%
%%%%%%%%%%%%%%%%%%%%%%%%%%%%%%%%%%%%%%%%

\section{Error expansion and  moment generating function of  $\max_{1\le k\le n}|X^{mn}_{t_k}-X^n_{t_k}|^2 $  }\label{secu}
For $(n,m)\in {\mathbb N}^*\times \bar{\mathbb N}$, we consider the Euler scheme $X^{n}$ on the grid $(t_k=\frac{kT}{n})_{0\le k\le n}$ and the process $X^{mn}$ which is the Euler scheme on the finer grid $(\frac{jT}{mn})_{0\le j\le mn}$ when $m$ is finite and the solution to \eqref{1} when $m=\infty$.  
We introduce the difference $U_{t_{k}}=X^{mn}_{t_{k}}-X^{n}_{t_{k}}$ between the two processes and define $U^\star_T=\max_{0\le k\le n}|U_{t_{k}}|$. For $s\in[0,T]$, we set $$\eta_l(s)=\left\lfloor \frac{ls}{T}\right\rfloor \frac{T}{l},\;\hat\eta_l(s)=\left\lceil \frac{ls}{T}\right\rceil \frac{T}{l}\mbox{ when }l\in{\mathbb N}^*\mbox{ and }\eta_\infty(s)=\hat\eta_\infty(s)=s.$$ We have $U_0=0$ and 
\begin{align*}
 \forall k\in\{0,\hdots,n-1\},\;U_{t_{k+1}}
&=U_{t_{k}}+\int_{t_k}^{t_{k+1}} b(X^{mn}_{\eta_{mn}(s)})ds-\frac{T}{n}b(X^n_{t_k})\\&=U_{t_k}+
\frac{T}{n}\left(b(X^{mn}_{t_k})-b(X^n_{t_k})\right)+\int_{t_k}^{t_{k+1}}\left(b(X^{mn}_{\eta_{mn}(s)})-b(X^{mn}_{t_k})\right)ds\end{align*}
To deal with this induction equation, it is convenient to introduce the matrices 
$$A_{k}=I_d+\mathbf 1_{\{U_{t_{k}}\neq 0\}}\frac{T}{n}\times\frac{\left(b(X^{mn}_{t_{k}})
-b(X^{n}_{t_{k}})\right)U_{t_{k}}^{\tp}}{\left|U_{t_{k}}\right|^2},$$
\begin{equation*}
\forall l\le k,\;\mathcal A^{l}_{k}:= A_{k}A_{k-1}\hdots A_{l}\mbox{ and  }\mathcal A^{k+1}_{k}=I_d.
\end{equation*}

This way, we have \begin{equation}
   U_{t_{k+1}}
=A_{k}U_{t_{k}}+\int_{t_k}^{t_{k+1}}\left(b(X^{mn}_{\eta_{mn}(s)})-b(X^{mn}_{t_k})\right)ds\label{recurs}.
\end{equation}
Let us introduce $V_{t_{k}}=\int_0^{t_k} \left(b(X^{mn}_{\eta_{mn}(s)})-b(X^{mn}_{\eta_{n}(s)})\right)ds$. One can check by induction that
\begin{equation}
   U_{t_{k}}=\sum_{l=1}^{k-1} \mathcal A_{k-1}^{l+1}(A_{l}-I_d)V_{t_{l}}+V_{t_{k}}.\label{formu}
\end{equation}
Indeed, since
\begin{align*}
   &\sum_{l=1}^k \mathcal A_{k}^{l+1}(A_{l}-I_d)V_{t_{l}}+V_{t_{k+1}}=(A_{k}-I_d)V_{t_{k}}+A_{k}\bigg(\sum_{l=1}^{k-1} \mathcal A_{k-1}^{l+1}(A_{l}-I_d)V_{t_{l}}+V_{t_{k}}\bigg)-A_{k}V_{t_{k}}+V_{t_{k+1}}\\
&\phantom{\sum}=A_{k}\bigg(\sum_{l=1}^{k-1} \mathcal A_{k-1}^{l+1}(A_{l}-I_d)V_{t_{l}}+V_{t_{k}}\bigg)+V_{t_{k+1}}-V_{t_{k}}
\end{align*}
and $V_{t_{k+1}}-V_{t_{k}}=\int_{t_k}^{t_{k+1}} \left(b(X^{mn}_{\eta_{mn}(s)})-b(X^{mn}_{\eta_{n}(s)})\right)ds$ both sides of \eqref{formu} statisfy the induction equality \eqref{recurs}. 
By It\^o's formula and the integration by parts formula, for $k\in\{1,\hdots,n\}$,
$$V_{t_k}=\int_0^{t_k}\gamma(s)\left((\nabla b(X^{mn}_s)b(X^{mn}_{\eta_{mn}(s)})+\frac{\Delta b}{2}(X^{mn}_s))ds+\nabla b(X^{mn}_s)dW_s\right)$$
where $\gamma(s)=(\hat{\eta}_n(s)-\hat{\eta}_{mn}(s))$.
Therefore
\begin{align}
   &U_{t_k}=(U\ta_{t_k}+U\tb_{t_k})\mbox{ with }\notag\\&U\ta_{t_k}=\sum_{l=1}^{k-1} \mathcal A^{l+1}_{k-1}(A_{l}-I_d)\int_0^{t_{l}}\gamma(t)\nabla b(X^{mn}_t)dW_t+\int_0^{{t_k}}\gamma(t)\nabla b(X^{mn}_t)dW_t\label{defu-1}\\
&\mbox{ and }U\tb_{t_k}=\sum_{l=1}^{k-1} \mathcal A^{l+1}_{k-1}(A_{l}-I_d)\int_0^{t_{l}}\gamma(t)(\nabla b(X^{mn}_t)b(X^{mn}_{\eta_{mn}(t)})+\frac{1}{2}\Delta b(X^{mn}_t))dt\notag\\&\phantom{\mbox{ and }U^2_{t_k}=}+\int_0^{{t_k}}\gamma(t)(\nabla b(X^{mn}_t)b(X^{mn}_{\eta_{mn}(t)})+\frac{1}{2}\Delta b(X^{mn}_t))dt\notag
\end{align}
respectively giving the contributions of the stochastic and the deterministic integrals. One can take advantage of the simpler expression $$U_{t_{k}}=\sum_{l=1}^{k-1} \mathcal A_{k-1}^{l+1}\int_{t_l}^{t_{l+1}}\left(b(X^{mn}_{\eta_{mn}(s)})-b(X^{mn}_{\eta_{n}(s)})\right)ds$$ (also proved by induction) and the linearity of the decomposition into stochastic (w.r.t. $dW_t$) and standard (w.r.t. $dt$) integrals to rewrite 
\begin{equation*}
   U\tb_{t_k}=\int_0^{t_k}\mathcal A^{\lceil\frac{nt}{T}\rceil}_{k-1}\gamma(t)(\nabla b(X^{mn}_t)b(X^{mn}_{\eta_{mn}(t)})+\frac 12\Delta b(X^{mn}_t))dt
\end{equation*} but the analogous expression of $U^1_T$ does not make sense because of the non-adapted factor $\mathcal A_{k-1}^{\lceil\frac{nt}{T}\rceil}$. This is the reason why we introduced the more complicated decomposition \eqref{formu}. Notice however than $n>T[\dot b]_\infty$, then the matrices $A_k$ and therefore $\mathcal A^1_{k}$ are invertible for $k\in\{0,\hdots,n\}$ so that $U\ta_{t_k}=\mathcal A^{1}_{k-1}\int_0^{t_k}\mathcal (\mathcal A^{1}_{\lfloor\frac{nt}{T}\rfloor})^{-1}\gamma(t)\nabla b(X^{mn}_t)dW_t$.
With Lemma \ref{aainv} just below, we conclude that
\begin{align}
   U^\star_T&\le U\taa^\star+U\tba^\star\mbox{ where }\notag\\
U\taa^\star&=\frac{[\dot b]_\infty T}{n}\sum_{k=1}^{n-1}e^{[\dot b]_\infty(T-t_{k+1})}\left|\int_0^{t_{l}}\gamma(t)\nabla b(X^{mn}_t)dW_t\right|+\max_{1\le k\le n}\left|\int_0^{{t_k}}\gamma(t)\nabla b(X^{mn}_t)dW_t\right|\label{defu1}\\
U\tba^\star&=\int_0^{T}e^{[\dot b]_\infty(T-t)}\gamma(t)\left|\nabla b(X^{mn}_t)b(X^{mn}_{\eta_{mn}(t)})+\frac 12\Delta b(X^{mn}_t)\right|dt.\label{defu2}
\end{align}

Let us state two lemmas that will be used to deal with $U\taa^\star$ and $U\tba^\star$ in the proof of Theorem \ref{laplace1}. The first one follows from usual linear algebra arguments and the submultiplicative property of the matrix norm. 
\begin{lemma}\label{aainv} One has  $\forall k\in\{1,\hdots,n-1\}$,
\begin{equation}\label{ub_A}
\|A_{k} -I_d\|\leq \frac{T[\dot b]_\infty}{n},\;\;\|A_{k}\|\leq 1+\frac{T[\dot b]_\infty}{n}\;\;\mbox{ and }\;\;\forall l\in\{0,\hdots,k\},\;\|\mathcal A_k^{l+1}\|\leq e^{[\dot b]_\infty (t_{k}-t_l)}.
\end{equation}
\end{lemma}

\begin{lemma}\label{momXn} When $b$ is Lipschitz continuous with constant $[\dot b]_\infty\in (0,+\infty)$, one has
\begin{enumerate}
\item 
   $\forall t\in[0,T],\;\sup_{n\in\bar{\mathbb N}}\sup_{s\le t}|X^n_s-x_0|\le e^{[\dot b]_\infty t}\sup_{s\le t}|W_s|+\frac{e^{[\dot b]_\infty t}-1}{[\dot b]_\infty}|b(x_0)|$,
   \item $\sup_{n\in\bar{\mathbb N}}\sup_{s\le t}|b(X^n_s)|\le e^{[\dot b]_\infty t}\left([\dot b]_\infty\sup_{s\le t}|W_s|+|b(x_0)|\right).$
\item Moreover, under ${\mathbf (\mathbf R\mathbf 1)}$, we have
$$\sup_{n\in\bar{\mathbb N}}\sup_{s\le t}|\Delta b(X^n_s)|\le 2a_{\Delta b}\left(e^{[\dot b]_\infty t}\sup_{s\le t}|W_s|+\frac{e^{[\dot b]_\infty t}-1}{[\dot b]_\infty}|b(x_0)|+1\right)$$
\item and under ${\mathbf (\mathbf R\mathbf 2)}$, we have
$$\sup_{n\in\bar{\mathbb N}}\sup_{s\le t}|\nabla\Delta b(X^n_s)|\le 2a_{\nabla\Delta b}\left(e^{[\dot b]_\infty t}\sup_{s\le t}|W_s|+\frac{e^{[\dot b]_\infty t}-1}{[\dot b]_\infty}|b(x_0)|+1\right).$$
\end{enumerate}
\end{lemma}
Notice that the function $b$ is Lipschitz continuous with constant $[\dot b]_\infty$ under ${\mathbf (\mathbf R\mathbf 1)}$.\begin{proof}For $n\in\bar{\mathbb N}$, one has 
$$\forall t\in[0,T],\;X^n_t-x_0=W_t+\int_0^t(b(X^n_{\eta_n(r)})-b(x_0))+b(x_0)dr.$$
Since $b$ is Lipschitz continuous with constant $[\dot b]_\infty$,

\begin{equation}
   \label{lipb}\forall x\in\R^d,\;|b(x)|\le|b(x)-b(x_0)| +|b(x_0)|\le[\dot b]_\infty|x-x_0|+|b(x_0)|.
\end{equation}One deduces that 
$$\sup_{s\le t}|X^n_s-x_0|\le \sup_{s\le t}|W_s|+[\dot b]_\infty\int_0^t
\sup_{s\le r}|X^n_s-x_0|+\frac{|b(x_0)|}{[\dot b]_\infty}dr.$$
Applying Gronwall's lemma to the function $t\mapsto\sup_{s\le t}|X^n_s-x_0|+\frac{|b(x_0)|}{[\dot b]_\infty}$, one obtains that
$$\forall t\in[0,T],\;\sup_{s\le t}|X^n_s-x_0|\le e^{[\dot b]_\infty t}\sup_{s\le t}|W_s|+\frac{e^{[\dot b]_\infty t}-1}{[\dot b]_\infty}|b(x_0)|.$$
The second (resp. third and fourth)  inequality follows by using \eqref{lipb} (resp. \eqref{afflapb} and \eqref{affgradlapb}). 
\end{proof}
\begin{proof}[Proof of Theorem \ref{laplace1}]
Let $x\ge 0$. Since $\max_{1\le k\le n}|U_{t_k}|=U^\star_T\le U\taa^\star+U\tba^\star$, Jensen's inequality implies that 
\begin{equation*}
   \left(\frac{(m-1)T^2x}{2mn}+\max_{1\le k\le n}|U_{t_k}|\right)^2\le \frac{(U\taa^\star)^2}{q}+\frac{(\frac{(m-1)T^2x}{2mn}+U\tba^\star)^2}{1-q}\end{equation*} where $q\in (0,1)$ is a parameter to be optimized later. With H\"older's inequality, we deduce that
\begin{align}
   \EE\left[  \exp\left\{ \rho \bigg(\frac{(m-1)T^2x}{2mn}+U^\star_T\bigg)^2\right\}\right]\leq \EE^q\left[ \exp\left\{\frac{\rho(U\taa^\star)^2}{q ^2}\right\}\right] \EE^{1-q}\left[ \exp\left\{\frac{\rho(\frac{(m-1)T^2x}{2mn}+U^\star\tba)^2}{(1-q)^2}\right\}\right].\label{UB:total}\end{align}

$\bullet\,$ {\bf First term}. Let us first deal with the contribution of $U\taa^\star$. Let us introduce the quantities  
\begin{equation}
   C=\sqrt{T}+\sum_{k=1}^{n-1}\frac{T[\dot b]_\infty e^{[\dot b]_\infty(T-t_{k+1})}\sqrt{t_{k}}}{n}\le \sqrt{T}+[\dot b]_\infty\int_0^Te^{[\dot b]_\infty(T-t)}\sqrt{t} dt:=C_{\eqref{cte:maj}},\label{cte:maj}
\end{equation}
 $p_n=\frac{\sqrt{T}}{C}$ and $p_{k}=\frac{[\dot b]_\infty T e^{[\dot b]_\infty(T-t_{k+1})}\sqrt{t_{k}}}{Cn}$ for $1\le k\le n-1$. 
Notice that $\sum_{k=1}^{n}p_{k}=1$ so that we have defined a probability measure. By \eqref{defu1},
\begin{align*}
 U^\star\taa&=C\sum_{k=1}^{n-1} \frac{p_k}{\sqrt{t_k}}\left|\int_0^{t_{k}}\gamma(t)\nabla b(X^{mn}_t)dW_t\right|+\frac{Cp_n}{\sqrt{T}}\max_{1\le k\le n}\left|\int_0^{t_k}\gamma(t)\nabla b(X^{mn}_t)dW_t\right|\\
&\le C_{\eqref{cte:maj}}\sum_{k=1}^{n} \frac{p_k}{\sqrt{t_k}}\max_{1\le l\le k}\left|\int_0^{t_{l}}\gamma(t)\nabla b(X^{mn}_t)dW_t\right|.
\end{align*}
Applying Jensen's inequality to the convex function $\R\ni x\mapsto \exp\left\{\frac{\rho x^2}{q ^2}\right\}$, we deduce that 
\begin{align*}
   \EE\left[  \exp\left\{\frac{\rho}{q^2}(U\taa^\star)^2\right\}\right]
\le \EE\bigg[ \sum_{k=1}^np_{k}\exp\left\{\frac{ C_{\eqref{cte:maj}}^2\rho}{q ^2t_{k}}\max_{1\le l\le k}\bigg|\int_0^{t_{l}}\gamma(t)\nabla b(X^{mn}_t)dW_t\bigg|^2\right\}\bigg].\end{align*}
Now, using the periodicity of the function $\gamma$ with period $t_1=T/n$, we get
\begin{align*}
\int_0^{t_{k}}\gamma(t)^2\Tr\left[\nabla b(X^{mn}_t) \nabla b(X^{mn}_t)^{\tp}\right]dt&\leq d [\dot b]^2_\infty\int_0^{t_{k}}\gamma^2(t)dt=d [\dot b]^2_\infty t_kT^2\frac{(2m-1)(m-1)}{6(mn)^2}.
\end{align*}
Then, by the second assertion in Lemma \ref{Laplace_G_square},
\begin{align}
 \forall \rho\in \bigg[0,&\frac{3(q mn)^2}{4C_\eqref{cte:maj}^2 T^2 d[\dot b]^2_\infty (2m-1)(m-1) }\bigg],\notag\\
  &\EE\left[  \exp\left\{\frac{\rho(U\taa^\star)^2}{q^2}\right\}\right]\le
 \exp\left\{\frac{2\ln(2) C_\eqref{cte:maj}^2\rho T^2 d[\dot b]^2_\infty (2m-1)(m-1)}{3(q mn)^2 }\right\}.\label{UB:1}
\end{align}
$\bullet\,$ {\bf Second term}.  
On the other hand, by \eqref{defu2} and Lemma \ref{momXn}, one has
\begin{align*}
U\tba^\star&\le \int_0^T\gamma(t)e^{[\dot b]_\infty(T-t)}\bigg(e^{[\dot b]_\infty t}([\dot b]^2_\infty+a_{\Delta b})\bigg(\sup_{s\le t}|W_s|+\frac{|b(x_0)|}{[\dot b]_\infty}\bigg)+a_{\Delta b}\bigg)dt.\end{align*}
Using $\int_0^T\gamma(t)dt=\frac{(m-1)T^2}{2mn}$ and $\tilde x=\frac{[\dot b]_\infty e^{-[\dot b]_\infty T}}{([\dot b]^2_\infty+a_{\Delta b})}x$, one deduces that
\begin{align*}
&\frac{(m-1)T^2x}{2mn}+U\tba^\star\le e^{T[\dot b]_\infty}([\dot b]^2_\infty+a_{\Delta b})\int_0^T\gamma(t)\bigg(\sup_{s\le t}|W_s|+\frac{|b(x_0)|+[\dot b]_\infty+\tilde x}{[\dot b]_\infty}\bigg)dt\notag\\
&=e^{T[\dot b]_\infty}([\dot b]^2_\infty+a_{\Delta b})\left(\int_0^T\gamma(t)\sup_{s\le t}|W_s|dt+\frac{(m-1)T^2}{2mn}\times\frac{|b(x_0)|+[\dot b]_\infty+\tilde x}{[\dot b]_\infty}\right)\\
&=e^{T[\dot b]_\infty}([\dot b]^2_\infty+a_{\Delta b})\int_0^T\gamma(t)\left(\sup_{s\le t}|W_s|+\frac{(m-1)T^2\sqrt{t}}{2mn\int_0^T\sqrt{r}\gamma(r)dr}\times\frac{|b(x_0)|+[\dot b]_\infty+\tilde x}{[\dot b]_\infty}\right)dt,
\end{align*}
where we used the periodicity of $\gamma(t)$ for the first equality. Setting $C_{\eqref{maju2}}=e^{T[\dot b]_\infty}([\dot b]^2_\infty+a_{\Delta b})$ and using Jensen's inequality for the probability density $p(t)=\frac{\sqrt{t}\gamma(t)}{\int_0^T\sqrt{s}\gamma(s)ds}$ on $[0,T]$, we obtain that
\begin{align}
   \EE&\left[  \exp\left\{\frac{\rho(\frac{(m-1)T^2x}{2mn}+U\tba^\star)^2}{(1-q)^2}\right\}\right]\le\int_0^T\EE\bigg[\exp\bigg\{\frac{\rho C^2_\eqref{maju2}(\int_0^T\sqrt{r}\gamma(r)dr)^2}{(1-q)^2}\notag\\&\;\;\;\;\;\;\;\;\;\;\;\;\;\;\;\;\;\;\;\;\;\left(\frac{1}{\sqrt{t}}\sup_{s\le t}|W_s|+\frac{(m-1)T^2}{2mn\int_0^T\sqrt{r}\gamma(r)dr}\times\frac{|b(x_0)|+[\dot b]_\infty+\tilde x}{[\dot b]_\infty}\right)^2\bigg\}\bigg]p(t)dt.\label{maju2}
\end{align}
Since $\sup_{s\le 1}|W_s|\le\sqrt{\sum_{i=1}^d\sup_{s\le 1}|W^i_s|^2}\le \sum_{i=1}^d\sup_{s\le 1}|W^i_s|$, for $\delta\ge 0$, 
\begin{align*}
   \left(\sup_{s\le 1}|W_s|+\delta\right)^2\le\sum_{i=1}^d\sup_{s\le 1}|W^i_s|^2+2\delta \sum_{i=1}^d\sup_{s\le 1}|W^i_s|+\delta^2\le \sum_{i=1}^d \left(\sup_{s\le 1}|W^i_s|+\delta\right)^2-(d-1)\delta^2,
\end{align*}
where the random variables in the sum in the right-hand side are independent. Setting $\delta=\frac{(m-1)T^2}{2mn\int_0^T\sqrt{r}\gamma(r)dr}\times\frac{|b(x_0)|+[\dot b]_\infty+\tilde x}{[\dot b]_\infty}$, plugging this inequality in \eqref{maju2} after using the scaling property of the Brownian motion $W$, we deduce that
\begin{align*}
   \EE\left[  \exp\left\{\frac{\rho (\frac{(m-1)T^2x}{2mn}+U\tba^\star)^2}{(1-q)^2}\right\}\right]\le &\exp\bigg\{-\frac{(d-1)(m-1)^2\rho C^2_\eqref{maju2} T^4(|b(x_0)|+[\dot b]_\infty+\tilde x)^2}{4((1-q)mn)^2[\dot b]^2_\infty}\bigg\}\\&\times\EE^d\bigg[\exp\bigg\{\frac{\rho C^2_\eqref{maju2} (\int_0^T\sqrt{r}\gamma(r)dr)^2}{(1-q)^2}\left(\sup_{s\le 1}|W^1_s|+\delta\right)^2\bigg\}\bigg].
\end{align*}
 Using the fact that for each $k\in\{1,\hdots,n\}$,  $r\mapsto \gamma(r)$ is non-increasing on $[t_{k-1},t_k]$ while $r\mapsto\sqrt{r}$ is increasing, we obtain that \begin{align*}
   \int_0^T\sqrt{r}\gamma(r)dr&=\sum_{k=1}^n\int_{t_{k-1}}^{t_k}\sqrt{r}\gamma(r)dr\le \sum_{k=1}^n\frac{n}{T}\int_{t_{k-1}}^{t_k}\gamma(r)dr\int_{t_{k-1}}^{t_k}\sqrt{r}dr\\&=\frac{(m-1)T}{2mn}\int_0^T\sqrt{r}dr=\frac{(m-1)T^{5/2}}{3mn}.
\end{align*}

Applying the first assertion in Lemma \ref{Laplace_G_square} with $|H|=1$ and $\mu=\frac{\rho C^2_\eqref{maju2} (\int_0^T\sqrt{r}\gamma(r)dr)^2}{(1-q)^2}$, we deduce that if $0\le \rho\le \frac{9((1-q)mn)^2}{8C^2_\eqref{maju2} T^5(m-1)^2}$, then 
\begin{align}
   \EE&\left[\exp\left\{\frac{\rho(\frac{(m-1)T^2x}{2mn}+U\tba^\star)^2}{(1-q)^2}\right\}\right] \notag\\&\le\exp\left\{\frac{\rho C^2_\eqref{maju2} T^4(m-1)^2}{((1-q)mn)^2}\bigg(\frac{(3d+1)(|b(x_0)|+[\dot b]_\infty+\tilde x)^2}{4[\dot b]^2_\infty}+\frac{4d\sqrt{T}(|b(x_0)|+[\dot b]_\infty+\tilde x)}{3\sqrt{\pi}[\dot b]_\infty}+\frac{4dT\ln 2}{9}\bigg)\right\}.\label{UB:2}
\end{align}
$\bullet\,$ {\bf Conclusion}. We now choose $q=\frac{C_\eqref{cte:maj}[\dot b]_\infty \sqrt{3d(2m-1)/m}}{C_\eqref{cte:maj} [\dot b]_\infty\sqrt{3d(2m-1)/m}+C_\eqref{maju2}T^{3/2}\sqrt{2(m-1)/m}}$ to obtain the same constraint on $\rho$ for the two terms $U\ta_T$ and $U\tb_T$ and conclude by combining \eqref{UB:total}, \eqref{UB:1} and \eqref{UB:2} that if $0\le \rho\le \frac{9mn^2}{4T^2(m-1)\left(C_\eqref{cte:maj} [\dot b]_\infty\sqrt{3d(2m-1)/m}+C_\eqref{maju2}T^{3/2}\sqrt{2(m-1)/m}\right)^2}$,
\begin{align*}\EE\left[\exp\left\{ \rho (U_T^\star)^2\right\}\right]\le \exp\left\{\rho C_{\eqref{majufin}}(x)\frac{(m-1)T^2}{mn^2}\right\}.\end{align*}   
\end{proof}
\begin{remark}
\begin{itemize}
   \item When $n>[\dot b]_\infty T$, one could consider using the alternative expression 
$U\ta_{t_k}=\mathcal A^{1}_{k-1}\int_0^{t_k}\mathcal (\mathcal A^{1}_{\lfloor\frac{nt}{T}\rfloor})^{-1}\gamma(t)\nabla b(X^{mn}_t)dW_t$ to replace in the first step $U\taa^\star$ by
$$e^{[\dot b]_\infty T}\max_{1\le k\le n}\left|\int_0^{t_k}\mathcal (\mathcal A^{1}_{\lfloor\frac{nt}{T}\rfloor})^{-1}\gamma(t)\nabla b(X^{mn}_t)dW_t\right|.$$
Lemma \ref{aainv}, implies that for $k\in\{0,\hdots,n\}$, $\|(\mathcal A^{1}_{k})^{-1}\|\le (1-\frac{[\dot b]_\infty T}{n})^{-k}$. This leads to replace $C_{\eqref{cte:maj}}$ by some constant not smaller than 
$$e^{[\dot b]_\infty T}\left(\int_0^Te^{2[\dot b]_\infty t}dt\right)^{1/2}\ge e^{[\dot b]_\infty T}\sqrt{T}=\sqrt{T}\left(1+[\dot b]_\infty\int_0^Te^{[\dot b]_\infty(T- t)}dt\right)\ge C_{\eqref{cte:maj}}.$$
\item In the last step of the derivation of \eqref{UB:total}, one could choose a constant $\tilde{q}\in (0,1)$ different from $q$ to obtain
\begin{align*}
   \EE\left[\exp\left\{ \rho (U_T^\star)^2\right\}\right]\leq \EE^{\tilde{q}}\left[ \exp\left\{\frac{\rho (U^\star\taa)^2}{q\tilde{q}}\right\}\right] \EE^{1-\tilde{q}}\left[ \exp\left\{\frac{\rho (U^\star\tba)^2}{(1-q)(1-\tilde{q})}\right\}\right],\end{align*}
but, following the reasoning in the above proof, this leads to the same upper-bound but under a stronger constraint on $\rho$. Indeed, for a fixed value of $q\tilde{q}$, the maximal value of $(1-q)(1-\tilde{q})=1-(q+\tilde{q})+q\tilde{q}$ is attained for $q=\tilde{q}$.
\item On the other hand, if $q(t)$ is some probability density on $[0,T]$ and $(Y_t)_{t\in[0,T]}$ an $\R^d$-valued process, applying Jensen's inequality to the convex function $\R^d\ni x\mapsto \exp\left\{|x|^2\right\}$ leads to
$$\E\left[\exp\left\{\left|\int_0^TY_tq(t)dt\right|^2\right\}\right]\le \int_0^T\E\left[\exp\{|Y_t|^2\}\right]q(t)dt.$$
whereas applying Jensen's inequality to $\R^d\ni x\mapsto|x|^2$ then H\"older's inequality leads to the upper-bound $\exp\left\{\int_0^T\ln\E\left[\exp\{|Y_t|^2\}\right]q(t)dt\right\}$ which is smaller when $\E\left[\exp\{|Y_t|^2\}\right]$ is not constant $dt$ a.e.. Nevertheless, in the above proof, the repeated use of Jensen's inequality for $x\mapsto \exp\left\{(x)^2\right\}$ did not worsen the final estimation because this estimation relies on uniform in $t\in[0,T]$ bounds for $\E\left[\exp\{|Y_t|^2\}\right]$ . 
\end{itemize}
\end{remark}

\begin{proof}[Proof of Proposition \ref{properfort}]
By \eqref{defu-1} and Lemma \ref{aainv}, $$|U\ta_{T}|\le \frac{[\dot b]_\infty T}{n}\sum_{k=1}^{n-1}e^{[\dot b]_\infty(T-t_{k+1})}\left|\int_0^{t_{l}}\gamma(t)\nabla b(X^{mn}_t)dW_t\right|+\left|\int_0^{T}\gamma(t)\nabla b(X^{mn}_t)dW_t\right|.$$ The difference between the right-hand side and the definition \eqref{defu1} of $U^\star\taa$ is that in the latter $\left|\int_0^{T}\gamma(t)\nabla b(X^{mn}_t)dW_t\right|$ is replaced by $\max_{1\le k\le n}\left|\int_0^{t_k}\gamma(t)\nabla b(X^{mn}_t)dW_t\right|$.
Reasoning like in estimation of the first term in the proof of Theorem \ref{laplace1}, we obtain that for the probability measure $(p_k)_{1\le k\le n}$ introduced in this estimation, 
\begin{align}
   \EE\left[|U^1_T |^2\right]&\le  C_{\eqref{cte:maj}}^2\sum_{k=1}^n\frac{p_{k}}{t_k}\EE\bigg[ \bigg|\int_0^{t_{k}}\gamma(t)\nabla b(X^{mn}_t)dW_t\bigg|^2\bigg]\notag\\&=C_{\eqref{cte:maj}}^2\sum_{k=1}^n\frac{p_{k}}{t_k}\int_0^{t_{k}}\gamma(t)^2\Tr\left[\nabla b(X^{mn}_t) \nabla b(X^{mn}_t)^{\tp}\right]dt\notag\\
&\le C_{\eqref{cte:maj}}^2[\dot b]_\infty^2 d T^2\frac{(2m-1)(m-1)}{6(mn)^2}\label{estieu1}\end{align}
When estimating $\EE\left[(U^\star\taa)^2\right]$ one needs to apply Doob's inequality to deal with the last and different term. Modifying the probability by giving weights proportional to $2\sqrt{T}$ for $k=n$ and to $\frac{[\dot b]_\infty T e^{[\dot b]_\infty(T-t_{k+1})}\sqrt{t_{k}}}{n}$ for $1\le k\le n-1$  leads to
\begin{equation}
   \EE\left[(U^\star\taa)^2\right]\le  (C_{\eqref{cte:maj}}+\sqrt{T})^2[\dot b]_\infty^2 d T^2\frac{(2m-1)(m-1)}{6(mn)^2}.\label{estieustar1}
\end{equation}
On the other hand, by the Minkowski inequality and \eqref{defu2},
\begin{equation}
  \EE^{1/2}\left[(U\tba^\star )^2\right] \le e^{[\dot b]_{\infty}T}\left(\int_0^Te^{-[\dot b]_{\infty}s}\gamma(s)\E^{1/2}\left[\left|\nabla b(X^{mn}_s)b(X^{mn}_{\eta_{mn}(s)})+\frac{\Delta b}{2}(X^{mn}_s)\right|^2\right]ds\right)
\end{equation}
Using Lemma \ref{momXn} and that $\E^2\left[\sup_{r\le s}|W_r|\right]\le\E\left[\sup_{r\le s}W_r^2\right]\le \E\left[\sum_{i=1}^d\sup_{r\le s}(W^i_r)^2\right]=4d\times s$, we have
\begin{align*}
 &e^{-[\dot b]_{\infty}s}\E^{1/2}\left[\left|\nabla b(X^{mn}_s)b(X^{mn}_{\eta_{mn}(s)})+\frac{\Delta b}{2}(X^{mn}_s))\right|^2\right]\\&\le[\dot b]_{\infty}\left(2[\dot b]_\infty\sqrt{d\times s}+|b(x_0)|\right)+a_{\Delta b}\left(2\sqrt{d\times s}+\frac{1-e^{-[\dot b]_\infty s}}{[\dot b]_\infty}|b(x_0)|+e^{-[\dot b]_\infty s}\right)\\
&\le |b(x_0)|\left(a_{\Delta b}\frac{1-e^{-[\dot b]_\infty s}}{[\dot b]_\infty}+[\dot b]_\infty\right)+2([\dot b]_{\infty}^2+a_{\Delta b})\sqrt{d\times s}+a_{\Delta b}e^{-[\dot b]_\infty s}:=g(s)+a_{\Delta b}e^{-[\dot b]_\infty s}
\end{align*}
Using the fact that for each $k\in\{1,\hdots,n\}$, $s\mapsto g(s)$ is non-decreasing on $[t_{k-1},t_k)$ while $s\mapsto\gamma(s)$ is non-increasing and bounded by $\frac{(m-1)T}{mn}$, we get that
\begin{align*}
   &e^{-[\dot b]_{\infty}T}\EE^{1/2}\left[(U\tba^\star)^2\right] \\&\le \sum_{k=1}^n\frac{n}{T}\int_{t_{k-1}}^{t_k}\gamma(s)ds\int_{t_{k-1}}^{t_k}g(s)ds+a_{\Delta b}\frac{(m-1)T}{mn}\int_0^Te^{-[\dot b]_\infty s}ds\\&=\frac{(m-1)T}{2mn}\int_0^Tg(s)ds+a_{\Delta b}\frac{(m-1)T}{mn}\times\frac{1-e^{-[\dot b]_\infty T}}{[\dot b]_\infty}\notag\\
&\le \frac{(m-1)T}{mn}\left(\frac{|b(x_0)|}{2}\bigg(a_{\Delta b}\frac{[\dot b]_\infty T-1+e^{-[\dot b]_\infty T}}{[\dot b]^2_\infty}+[\dot b]_\infty T\bigg)+a_{\Delta b}\frac{1-e^{-[\dot b]_\infty T}}{[\dot b]_\infty}+\frac{2}{3}\sqrt{d}([\dot b]_{\infty}^2+a_{\Delta b})T^{3/2}\right).
\end{align*}
The estimation of $\E\left[|U_T|^2\right]$ (resp. $\E\left[|U^\star_T|^2\right]$) is obtained by plugging this inequality together with \eqref{estieu1} (resp. \eqref{estieustar1}) into the inequality $$\E\left[|U_T|^2\right]\le \frac{\E\left[|U\ta_T|^2\right]}{q}+\frac{\E\left[(U\tba^\star)^2\right]}{1-q}\mbox{ resp. }\E\left[|U^\star_T|^2\right]\le \frac{\E\left[|U\taa^\star|^2\right]}{q}+\frac{\E\left[(U\tba^\star)^2\right]}{1-q}\mbox{ for } q\in (0,1)$$and optimizing over $q$ : for $a,b>0$, $\min_{q\in (0,1)}\frac{a}{q}+\frac{b}{1-q}=(\sqrt{a}+\sqrt{b})^2$ attained for $q=\frac{\sqrt{a}}{\sqrt{a}+\sqrt{b}}$. \end{proof}

%%%%%%%%%%%%%%%%%%%%%%%%%%%%%%%%%%%%%%%%%%%%%%%%%%%%%%%%%%%
%%%%%%%%%%%%%%%%%%%%%%%%%%%%%%%%%%%%%%%%%%%%%%%%%%%%%%%%%%%

\section{Error expansion and  moment generating function of its Malliavin derivative}
\subsection{Basic facts on Malliavin calculus}

In this work, we follow the notations, definitions and results of \cite{Nua}. Let  $(W_t)_ {0\leq t\leq T}$ be a $d$-dimensional Brownian motion defined on the filtered 
probability space  $(\Omega,\mathcal F,(\mathcal F_t)_{t\geq 0},\PP)$.  Let $\D$ denote the Malliavin derivative 
operator taking values in the real separable Hilbert space $H:=L^2([0,T],\RR^d)$ whose norm is denoted by $|\cdot|_H$. More precisely, for $h\in H$, we denote $W(h)$ the Wiener integral $W(h)=\int_0^Th(t)\cdot dW_t$. Let $\mathcal S$ denote the class of random variables of the form $F=f(W(h^1),\cdots, W(h^n))$, for $n\geq 1$, with $(h^1,\cdots, h^n)\in H^{\otimes n}$  and  $f\in\mathscr C^{\infty}_p(\RR^n,\RR)$. Then, for $F\in\mathcal S$, the Malliavin derivative of $F$ denoted $\D F=(\D^i_tF, 0\leq t\leq T ,1\leq i\leq d)$ is defined by
$$
\D^i_tF=\sum_{k=1}^n\frac{\partial f}{\partial x_k}(W(h^1),\cdots,W(h^n))h^k_i(t),
$$
where $h_i^k$ denotes the $i$-th coordinate of $h^k$.
The operator $\D$ is closable as an operator from $L_p(\Omega)$ to   $L_p(\Omega,H)$, for any $p\geq 1$. Its domain with respect to the norm $|F|_{1,p}:= [\mathbb E(|F|^p)+
\mathbb E(|\D F|^p_{H})]^p$ is denoted by 
 $\mathbb D^{1,p}$.

We now state some essential  properties, which are going to be useful in the sequel.
\begin{prop}[Chain's rule]
Let $\phi\in\mathscr C^{1}(\RR^q,\RR)$ with bounded first order derivatives and $F=(F^1,\cdots,F^q)$ be an $\RR^q$-valued random vector with $F^k\in\mathbb D^{1,p}$ for $k=1,\dots,q$. Then,
$\phi(F)\in \mathbb D^{1,p}$ and for each $i=1,\dots,d$
$$
\D_t^i\phi(F)=\sum_{j=1}^q\frac{\partial \phi}{\partial x^j}(F)\D_t^iF^j.
$$
\end{prop}
\begin{prop}[Clark-Ocone formula]
Let $F$ be a $\mathcal F_T$-measurable random variable that belongs to $\mathbb D^{1,p}$ for some $p\ge 1$. Then, 
$$
F=\mathbb E(F)+\sum_{i=1}^d\int_0^T\mathbb E(\D^i_s F|\mathcal F_s)dW^i_s,\quad \mbox{a.s.}
$$ 
\end{prop}
%%%%%%%%%%%%%%%%%%%%%%%%%%%%%%%%%%%%%%%%%%%%%%%%%%%%%%%%%%%
A preliminary essential result is on the boundedness of Malliavin's derivative of the diffusion and its Euler scheme given by \eqref{euler} 
\begin{lemma}
\label{boundMal}
Let $X = (X_t)_{ t \in[0, T]}$ be the solution to Eq. \eqref{1},
where the coefficient $b$ satisfies our  global Lipschitz condition  $(\mathcal H_{GL})$. Then, $X_t, X^n_t\in\mathbb D^{1,\infty}:=\cap_{p\geq 1}\mathbb D^{1,p}$ for any $t\in[0,T]$. 
Moreover, assume  that $b\in\mathscr C^1(\RR^d,\RR^d )$
with  $\|\nabla b\|\leq [\dot b]_{\infty}$ for some finite constant $[\dot b]_{\infty}$.  Then,   for all $1\leq j\leq d$ we have
 $$
\forall r,t\in[0,T],\;|\D^j_rX_{t} |\vee \sup_{n\in\NN^*}|\D^j_rX^{n}_{t} |) \leq \mathbf 1_{\{r\leq t\}} e^{[\dot b]_{\infty}(t-r)} .
$$
\end{lemma}
\begin{proof} 
Under our assumptions,  for any $t\in[0,T]$, the random variables  $X_t, X^n_t$ belong to $\mathbb D^{1,\infty}$ (see \cite{Nua}, Section 2.2).  The estimation of the Malliavin derivative of $X_t$ is straightforward. We only give a proof for the estimation of the Malliavin derivative of the Euler scheme. For $r,t\in[0,T]$,
\begin{equation}
   \D^j_rX^{n}_t=\begin{cases}
      0\mbox{ if }t<r,\\
\mathbf e_j\mbox{ if }r\le t\le \lceil \frac{rn}{T}\rceil\frac{ T}{n}\\
\mathbf e_j +\int_{ \lceil \frac{rn}{T}\rceil\frac{ T}{n} }^t \nabla b(X^{n}_{\eta_{{n}}(u)})\D^j_rX^{n}_{\eta_{{n}}(u)}du\mbox{ if }t\ge \lceil \frac{rn}{T}\rceil\frac{ T}{n}
   \end{cases},\label{dermal}
\end{equation}
where $(\mathbf e_j)_{j=1,\dots,d}$ denotes the canonical basis of $\RR^d$.
Hence $\D^j_rX^{n}_{\lceil \frac{rn}{T}\rceil\frac{T}{n}}=\mathbf e_j$ and for $t\ge\lceil \frac{rn}{T}\rceil\frac{ T}{n}$, 
$$\forall k\in\left\{ \left\lceil \frac{rn}{T}\right\rceil,\dots,\left\lfloor\frac{tn}{T}\right\rfloor \right\},\;\D^j_rX^{n}_{t\wedge \frac{(k+1)T}{n}}=\D^j_rX^{n}_{\frac{kT}{n}}+\left(t\wedge \frac{(k+1)T}{n}-\frac{kT}{n}\right)\nabla b(X^{n}_{\frac{kT}{n}})\D^j_rX^{n}_{\frac{kT}{n}}.$$
so that
$$
\D^j_rX^{n}_{t}=\left(I_d+(t-\eta_n(t))\nabla b(X^{n}_{\eta_n(t)})\right)
\prod^{\left\lceil \frac{rn}{T}\right\rceil}_{k=\left\lfloor\frac{tn}{T}\right\rfloor-1}
\left(I_d+\frac{T}{n}\nabla b(X^{n}_{\frac{kT}{n}})\right)\mathbf e_j.
$$
Using the boundedness of the  first order derivatives of $b$, we deduce that for $t\ge\lceil \frac{rn}{T}\rceil\frac{ T}{n}$,
$$
|\D^j_rX^{n}_{t} |\leq \left(1+(t-\eta_n(t))[\dot b]_{\infty} \right)\left(1+\frac{T}{n} [\dot b]_{\infty} \right)^{\left\lfloor\frac{tn}{T}\right\rfloor-\left\lceil \frac{rn}{T}\right\rceil}\leq e^{[\dot b]_{\infty}(t-\lceil \frac{rn}{T}\rceil\frac{ T}{n}) }.
$$
\end{proof}
\subsection{Moment generating function of~${\displaystyle \max_{1\le k\le n}}|\D^j_rX^{mn}_{t_k}-\D_r^j X^n_{t_k}|^2$}
The next theorem states an exponential type upper bound for  the moment generating function  of ${\displaystyle \max_{1\le k\le n}}|\D^j_rX^{mn}_{t_k}-\D_r^j X^n_{t_k}|^2$. In what follows, we refer to  constants notation {\bf (CN)} introduced in Section \ref{subsec:momgen}.
\begin{theorem}\label{laplace2}
Let assumption ${\mathbf (\mathbf R\mathbf 2)}$ hold, $(n,m)\in {\mathbb N}^*\times\bar{\mathbb N}$, $t_k=\frac{kT}{n}$ for $k\in\{0,\hdots,n\}$  and $\rho$  be a constant satisfying  $0\le \rho\leq e^{-2[\dot b]_\infty(T-r)}\hat\rho(r)n^2$. 
Then,
\begin{align}\label{Th_Laplace_DU}
\EE&\left[  \exp\left\{ \rho \max_{0\le k\le n}|\D^j_rX^{mn}_{t_k}-\D^j_rX^{n}_{t_k}|^2\right\}\right]\leq
 \exp\left\{\rho e^{2[\dot b]_\infty(T-r)}\Phi(r,[\dot b]_\infty)\frac{(m-1)T^2}{mn^2}\right\}.
\end{align}
\end{theorem}

\begin{proof}
Let $j\in\{1,\hdots,d\}$ and $r\in[0,T]$. By \eqref{dermal}, for $k\in\{0,\hdots,\lfloor\frac{rn}{T}\rfloor\}$, $\D^j_rU_{t_k}=0$, \begin{equation}
   \D^j_rU_{t_{{\lceil \frac{rn}{T}\rceil }}}= \int_{\hat\eta_{{mn}}(r)}^{\hat\eta_{{n}}(r)} \nabla b(X^{mn}_{\eta_{{mn}}(s)})\D^j_rX^{mn}_{\eta_{{mn}}(s)}ds,\label{djuinit}
\end{equation} and for $k\in\{\lceil \frac{rn}{T}\rceil,\hdots,n-1\}$, \begin{align*}
 \D^j_rU_{t_{k+1}} =&\D^j_rU_{t_{k}}+ \frac{T}{n}\nabla b(X^{mn}_{t_{k}})\D^j_rU_{t_{k}}+ \frac{T}{n}(\nabla b(X^{mn}_{t_{k}})-\nabla b(X^{n}_{t_{k}}))\D^j_rX^n_{t_k}\\&+\int_{t_k}^{t_{k+1}} \nabla b(X^{mn}_{_{\eta_{{mn}}(s)} })\D^j_rX^{mn}_{_{\eta_{{mn}}(s)} }-\nabla b(X^{mn}_{t_{k} })\D^j_rX^{mn}_{t_{k} }ds.
\end{align*}
Setting $$
B_{k}:=I_d+\frac{T}{n}\nabla b(X^{mn}_{t_{k}}),
$$
and defining $( V^{(r,j)}_{t_{k}}  )_{\lceil \frac{rn
}{T}\rceil\leq k\leq n}$ inductively by $V^{(r,j)}_{t_{{\lceil \frac{rn
}{T}\rceil }}}=0$ and 
\begin{align}\label{rec:v}
 V^{(r,j)}_{t_{k+1}}-V^{(r,j)}_{t_{k}}&=\frac{T}{n}\left(  \nabla b(X^{mn}_{t_k})-\nabla b(X^{n}_{t_k})\right)\D^j_rX^{n}_{t_{k}}\notag\\
&\phantom{\frac{T}{n}}+\int_{t_k}^{t_{k+1}} \nabla b(X^{mn}_{_{\eta_{{mn}}(s)} })\D^j_rX^{mn}_{_{\eta_{{mn}}(s)} } -\nabla b(X^{mn}_{t_{k} })\D^j_rX^{mn}_{t_k}ds,\end{align}
we deduce that for $k\in\{\lceil \frac{rn}{T}\rceil,\hdots,n-1\}$,
\begin{align*}\label{rec:uv}
\D^j_rU_{t_{k+1}} -V^{(r,j)}_{t_{k+1}} &=B_{ k} \D^j_rU_{t_{k}}- V^{(r,j)}_{t_{k}}=  B_{ k} (\D^j_rU_{t_{k}}-  V^{(r,j)}_{t_{k}})
 +(B_{ k}-I_d) V^{(r,j)}_{t_{k}},
\end{align*}
equality similar to \eqref{recurs}.
Let us introduce
\begin{equation*}
\forall l\le k,\;\mathcal B^{l}_{k}:= B_{k}B_{k-1}\hdots B_{l}\mbox{ and  }\mathcal B^{k+1}_{k}=I_d.
\end{equation*}
One can check by induction on $k$ that for all $k\in\{\lceil \frac{rn
}{T}\rceil
,\hdots,n\}$ and any sequence $(\tilde V^{(r,j)}_{t_{k}}  )_{\lceil \frac{rn
}{T}\rceil\leq k\leq n}$ such that $\tilde V^{(r,j)}_{t_{\lceil \frac{rn
}{T}\rceil}}=0$, we have 
\begin{align}\label{DU:res1}
 D^j_r  U_{t_{k}}=& \mathcal B_{k-1}^{{\lceil \frac{rn
}{T}\rceil } }D^j_rU_{t_{{\lceil \frac{rn
}{T}\rceil }}}+\sum_{l={\lceil \frac{rn
}{T}\rceil }}^{k-1} \mathcal B_{k-1}^{l+1}(\tilde V^{(r,j)}_{t_{l+1}}-\tilde V^{(r,j)}_{t_{l}})\notag\\&+\sum_{l={\lceil \frac{rn
}{T}\rceil }+1}^{k-1} \mathcal B_{k-1}^{l+1}(B_{l}-I_d)(V^{(r,j)}_{t_{l}}-\tilde{V}^{(r,j)}_{t_{l}})+V^{(r,j)}_{t_{k}}-\tilde V^{(r,j)}_{t_{k}}.
\end{align}
Let us explicit the right-hand side of \eqref{rec:v}. We have $$
\frac{T}{n} \left(  \nabla b(X^{mn}_{t_{k}})-\nabla b(X^{n}_{t_{k}})\right)\D^j_rX^{n}_{t_{k}}   =
\frac{T}{n} H^b_{t_{k}}U_{t_{k}}\D^j_rX^{n}_{t_{k}},
$$
where for $u\in \RR^d$,  the  $l^{\rm th}$ column of the $d\times d$ matrix  $(H^b_{t_{k_1}}u)$
is given by 
$$
\Bigl(H^b_{t_{k}}u \Bigr)_{\cdot l}:=\Bigl(\int_0^1
\frac{\partial\nabla b }{\partial x_l}(\theta X^n_{t_{k}} +(1-\theta)X^{mn}_{t_{k}})d\theta\Bigr)u.
$$
Then, we write
\begin{eqnarray*}\label{eq:split}
  \nabla b(X^{mn}_{{_{\eta_{{mn}}(s)} } })\D^j_rX^{mn}_{{_{\eta_{{mn}}(s)} } }- \nabla b(X^{mn}_{{_{\eta_{{n}}(s)} } })\D^j_rX^{mn}_{{_{\eta_{{n}}(s)} } } &=& 
(  \nabla b(X^{mn}_{{_{\eta_{{mn}}(s)} } })-  \nabla b(X^{mn}_{{_{\eta_{{n}}(s)} } })   )\D^j_rX^{mn}_{{_{\eta_{{n}}(s)} } }  \\ &+&
 \nabla b(X^{mn}_{{_{\eta_{{mn}}(s)} } })(\D^j_rX^{mn}_{{_{\eta_{{mn}}(s)} } }-\D^j_rX^{mn}_{{_{\eta_{{n}}(s)} } } )
\end{eqnarray*}
For the first term,  by It\^o's formula and the integration by parts formula, 

\begin{align*}
&\int^{t_{k+1}}_{t_{k}}(  \nabla b(X^{mn}_{{_{\eta_{{mn}}(s)} } })-  \nabla b(X^{mn}_{{_{\eta_{{n}}(s)} } })   )\D^j_rX^{mn}_{{_{\eta_{{n}}(s)} } }ds
\\&=\int^{t_{k+1}}_{t_{k}}\gamma(s)\left(\Bigl(\nabla^2 b(X^{mn}_s)b(X^{mn}_{\eta_{mn}(s)})+\frac{1}{2} \nabla[\Delta b(X^{mn}_s)]\Bigr)\D^j_rX^{mn}_{{_{\eta_{{n}}(s)} } } ds+
\nabla^2 b(X^{mn}_s)\D^j_rX^{mn}_{{_{\eta_{{n}}(s)} } }dW_s\right)
\end{align*}
where $\gamma(s)=(\hat{\eta}_n(s)-\hat{\eta}_{mn}(s))$ and for  $u\in \RR^d$, the $l^{\rm th}$ columns of the $d\times d$ matrices  $(\nabla^2 b(X^n_t)u)$ and $\Bigl(\nabla[\Delta b(X^n_t)] \Bigr)$
are given by
$$
\Bigl(\nabla^2 b(X^n_t)u\Bigr)_{\cdot l}:=\frac{\partial\nabla b(X^n_t)}{\partial x_l}
u\mbox{ and }\Bigl(\nabla[\Delta b(X^n_t)] \Bigr)_{\cdot l}:=\frac{\partial\Delta b(X^n_t)}{\partial x_l},\;\;\; 1\leq l\leq d.
$$

Concerning the second term we use \eqref{dermal} to  write 
$$
\D^j_rX^{mn}_{{_{\eta_{{mn}}(s)} } }-\D^j_rX^{mn}_{{_{\eta_{{n}}(s)} } }  =
\int_{ {\eta_{{n}}(s)} }^{{\eta_{{mn}}(s)}  }\nabla b(X^{mn}_{\eta_{mn}(u) })\D^j_rX^{mn}_{\eta_{mn}(u)}du+
1_{\{  {\eta_{{n}}(s)}  < r\leq {\eta_{{mn}}(s)} \}}\mathbf e_j.
$$
For $k\in\{\lceil \frac{rn}{T}\rceil,\hdots,n-1\}$,  remarking that $ \int^{t_{k+1}}_{t_{k}}\nabla b(X^{mn}_{{_{\eta_{{mn}}(s)} } })1_{\{  {\eta_{{n}}(s)} <r\leq {\eta_{{mn}}(s)} \}}\mathbf e_jds=0 $ and using Fubini's theorem, we  get
\begin{equation*}
 \int^{t_{k+1}}_{t_{k}} \nabla b(X^{mn}_{{_{\eta_{{mn}}(s)} } })(\D^j_rX^{mn}_{{_{\eta_{{mn}}(s)} } }-\D^j_rX^{mn}_{{_{\eta_{{n}}(s)} } })ds=
  \int^{t_{k+1}}_{t_{k}} \left(\int_{\hat{\eta}_{mn}(s) }^{\hat{\eta}_n(s) } \nabla b(X^{mn}_{{_{\eta_{{mn}}(u)} } })du\right) \nabla b(X^{mn}_{{_{\eta_{{mn}}(s)} } })
\D^j_rX^{mn}_{{_{\eta_{{mn}}(s)} } }ds.
\end{equation*}
Hence, 
\begin{align*}
& V^{(r,j)}_{t_{k+1}}-V^{(r,j)}_{t_{k}}=\int^{ t_{k+1}}_{t_{k}}\gamma(s)\nabla^2 b(X^{mn}_s)\D^j_rX^{mn}_{{_{\eta_{{n}}(s)} } }dW_s+\frac{T}{n} H^b_{t_{k}}U_{t_{k}}\D^j_rX^{n}_{t_{k}}+\int^{t_{k+1}}_{t_{k}}\gamma(s)G^{b,mn}_sds
\\
&\mbox{ with }G^{b,mn}_s:=\left(\nabla^2 b(X^{mn}_s)b(X^{mn}_{\eta_{mn}(s)})+\frac{1}{2} \nabla[\Delta b(X^{mn}_s)]\right)
\D^j_rX^{mn}_{{_{\eta_{{n}}(s)} } }
\\ 
&\phantom{\mbox{ with }G^{b,mn}_s:=}+ \left(\frac{ 1_{\{\gamma(s)>0\}} }{\gamma(s)}\int_{\hat{\eta}_{mn}(s) }^{\hat{\eta}_n(s) } \nabla b(X^{mn}_{{_{\eta_{{mn}}(u)} } })du\right)\nabla b(X^{mn}_{{_{\eta_{{mn}}(s)} } })
\D^j_rX^{mn}_{{_{\eta_{{mn}}(s)} } }.
\end{align*}
Choosing $\tilde{V}^{(r,j)}_{t_{k+1}}-\tilde V^{(r,j)}_{t_{k}}$ equal to the sum of the second and third terms in the above expression of $V^{(r,j)}_{t_{k+1}}-V^{(r,j)}_{t_{k}}$ and applying \eqref{DU:res1}, we conclude that for $k\in\{\lceil \frac{rn}{T}\rceil,\hdots,n\}$,
\begin{align} 
 \D^j_r  U_{t_{k}}&=\D^j_r  U^{(0)}_{t_{k}}+
\D^j_r  U^{(1)}_{t_{k}}+ \D^j_r  U^{(2)}_{t_{k}}+ \D^j_r  U^{(3)}_{t_{k}}\mbox{ with}\label{decdju}\\
\D^j_r  U^{(0)}_{t_{k}}&:=\mathcal B_{k-1}^{{\lceil \frac{rn
}{T}\rceil } }\D^j_rU_{t_{{\lceil \frac{rn
}{T}\rceil }}}\notag\\
 \D^j_r  U^{(1)}_{t_{k}}&:= \frac{T}{n}\sum_{l={\lceil \frac{rn
}{T}\rceil }}^{k-1} \mathcal B_{k-1}^{l+1}H^b_{t_{l}}U_{t_{l}}\D^j_rX^{n}_{t_{l}}\notag\\
\D^j_r  U^{(2)}_{t_{k}}&:=\sum_{l={\lceil \frac{rn
}{T}\rceil }}^{k-1} \mathcal B_{k-1}^{l+1}\int_{t_l}^{t_{l+1}}\gamma(s)G^{b,mn}_sds\notag\\
\D^j_r  U^{(3)}_{t_{k}}&:=  \sum_{l={\lceil \frac{rn
}{T}\rceil }+1}^{k-1} \mathcal B_{k-1}^{l+1}(B_{l}-I_d)\int_{ {\hat\eta_{{n}}(r)}   }^{t_{l}}\gamma(s)\nabla^2 b(X^{mn}_s)\D^j_rX^{mn}_{{_{\eta_{{n}}(s)} } }dW_s\notag\\
&\phantom{\sum_{l={\lceil \frac{rn
}{T}\rceil }}^{k-1} \mathcal B_{k-1}^{l+1}(B_{l}-I_d)\int_{ {\hat\eta_{{n}}(r)}   }^{t_{l}}\gamma(s)\nabla^2 b(X^{mn}_s)dW_s }
 +\int_{ {\hat\eta_{{n}}(r)}   }^{t_{k}}\gamma(s)\nabla^2 b(X^{mn}_s)\D^j_rX^{mn}_{{_{\eta_{{n}}(s)} } } dW_s.\notag
\end{align}
Combining Assumption ${\mathbf (\mathbf R\mathbf 2)}$, standard linear algebra arguments and Lemma \ref{boundMal}, we easily prove the following result.
\begin{lemma}\label{bbinv} One has  $\forall k\in\{1,\hdots,n-1\}$,
\begin{equation*}\label{ub_B}
\|B_{k} -I_d\|\leq \frac{T[\dot b]_\infty}{n},\;\;\|B_{k}\|\leq 1+\frac{T[\dot b]_\infty}{n},\quad \forall l\in\{0,\hdots,k\},\;\|\mathcal B_k^{l+1}\|\leq e^{[\dot b]_\infty (t_{k}-t_l)}
\end{equation*}
and $\forall s\in [{\hat\eta_{{n}}(r)} , T]$
\begin{equation*}\label{ub:G}
|G^{b,mn}_s|\leq  e^{[\dot b]_\infty (\eta_{{mn}}(s)-r)}\left( \sqrt{d} [\ddot b]_\infty |b(X^{mn}_{\eta_{mn}(s)})|+\frac{1}{2} \|\nabla[\Delta b(X^{mn}_s)]\|
+[\dot b]^2_\infty\right).
\end{equation*}
\end{lemma}
Combining this result, \eqref{djuinit} and  Lemma \ref{boundMal}, we easily get an upper bound for  $\D^j_r  U^{(0)}_{t_{k}}$
\begin{align}
\label{term_ini}
|\D^j_r  U^{(0)}_{t_{k}}| 
\leq \|\mathcal B_{k-1}^{{\lceil \frac{rn}{T}\rceil } } \| |\D^j_rU_{t_{{\lceil \frac{rn}{T}\rceil }}} |
\leq e^{[\dot b]_\infty (t_{k}-\hat\eta_n(r))}\times [\dot b]_\infty e^{[\dot b]_\infty(\hat\eta_n(r)-r)}\gamma(r)=[\dot b]_\infty e^{[\dot b]_\infty(t_k-r)}\gamma(r).
\end{align}

In the same way as in the previous section where we introduced $U^\star_T=\max_{0\le k\le n}|U_{t_{k}}|$, we also define for all $r\in[0,T]$, the process $\D_r^jU^\star_T=\max_{0\le k\le n}|\D_r^jU^{(0)}_{t_{k}}+\D_r^jU^{(1)}_{t_{k}}+\D_r^jU^{(2)}_{t_{k}}+\D_r^jU^{(3)}_{t_{k}}|$. According to our assumption  ${\mathbf (\mathbf R\mathbf 2)}$ and by 
Lemmas \ref{boundMal} and \ref{bbinv}, we have
\begin{align}
\D_r^jU^\star_T &\leq \D_r^jU^\star_{(0)} +\D_r^jU^\star_{(1)} + \D_r^jU^\star_{(2)} + \D_r^jU^\star_{(3)}, \mbox{ with }\notag\\
\D_r^jU^\star_{(0)} &\le[\dot b]_\infty e^{[\dot b]_\infty(T-r)}\frac{(m-1)T}{mn}\label{majodu0}\\ \D_r^jU^\star_{(1)}&:=\frac{T}{n} e^{[\dot b]_\infty (T-r)}\sum_{k={\lceil \frac{rn
}{T}\rceil }}^{n-1} e^{[\dot b]_\infty (T- t_{k+1})} \|H^b_{t_k}U_{t_k } \|\notag  \\
\D_r^jU^\star_{(2)}&=\int_{ {\hat\eta_{{n}}(r)}   }^{T} e^{[\dot b]_\infty (T-\hat\eta_{{n}}(s) )}|\gamma(s)G^{b,mn}_s|ds\notag\\
 \D^j_r  U_{(3)}^{\star}&:= \frac{T}{n} \sum_{k={\lceil \frac{rn
}{T}\rceil }+1}^{n-1}    {[\dot b]_\infty e^{[\dot b]_\infty (T-t_{k+1})}} |\int_{ {\hat\eta_{{n}}(r)}   }^{t_{k}}\gamma(s)\nabla^2 b(X^{mn}_s)\D^j_rX^{mn}_{{_{\eta_{{n}}(s)} } }dW_s|\notag\\
&+\max_{{\lceil \frac{rn
}{T}\rceil +1} \le k\le n}|\int_{ {\hat\eta_{{n}}(r)}   }^{t_{k}}\gamma(s)\nabla^2 b(X^{mn}_s)\D^j_rX^{mn}_{{_{\eta_{{n}}(s)} } }dW_s|. \label{DU2}
\end{align}

By Jensen's inequality we get 
$$
|\D_r^jU^\star_T|^2 \leq  
\frac{|\D_r^jU^\star_{(1)}|^2}{(1-q)} + \frac{|\D_r^jU^\star_{(0)}+\D_r^jU^\star_{(2)}|^2}{q (1-\bar q)} 
+ \frac{|\D_r^jU^\star_{(3)}|^2}{ q \bar q}
$$
where  $q,  \bar q\in(0,1)$ are two parameters to be optimized later. Then, by H\"older and Jensen inequalities, we deduce that for all $r\in[0,T]$,
we have
\begin{align}\label{Laplace_DU_2}
\EE\left[  \exp\left\{ \rho |\D^j_rU_T^{\star}|^2\right\}\right]&\leq
\EE^{ 1-q}\left[  \exp\left\{    \frac{ \rho }{  (1-q)^2}|\D_r^jU^\star_{(1)}|^2 \right\}\right]\notag\\&\times\EE^{q(1-\bar q)}\left[  \exp\left\{\frac{ \rho }{q^2(1  -\bar  q)^2}|\D_r^jU^\star_{(0)}+\D_r^jU^\star_{(2)}|^2 \right\}\right]\EE^{q \bar q }\left[  \exp\left\{\frac{ \rho }{q^2\bar q^2}|\D_r^jU^\star_{(3)}|^2 \right\}\right],
\end{align}
$\bullet\;${\bf First term}. 
In this part, we  focus on the contribution of the term $\D_r^jU^\star_{(1)}$. Note that under assumption ${\mathbf (\mathbf R\mathbf 2)}$
we have,  for all 
$k\in\{\lceil\frac{rn
}{T}\rceil , \cdots, n-1\}$
$$
\|H^b_{t_k}U_{t_k}  \| \leq \left(\sum_{j=1}^d 
\int_0^1
\left\|\frac{\partial\nabla b }{\partial x_j} (\theta X^n_{{t_k }} +(1-\theta)X^{mn}_{t_k}) \right\|^2 d\theta \,|U_{t_k}|^2\right)^{\frac{1}{2}}\leq \sqrt{d}[\ddot b]_\infty |U_{t_k}|
$$
Therefore,
\begin{align}\label{DU:1}
\D_r^jU^\star_{(1)}&\leq \frac{T \sqrt{d}[\ddot b]_\infty }{n} U^\star_T e^{[\dot b]_\infty (T-r)}\sum_{k={\lceil \frac{rn
}{T}\rceil }}^{n-1} e^{[\dot b]_\infty (T- t_{k+1})} \notag\\
&\leq { \sqrt{d}[\ddot b]_\infty } U^\star_Te^{[\dot b]_\infty (T-r)}\int_{\hat\eta_{n(r)}}^{T}e^{[\dot b]_\infty (T- t)}dt \notag\\
& \leq e^{[\dot b]_\infty (T-r)} \frac{\sqrt{d}[\ddot b]_\infty }{[\dot b]_\infty }(e^{[\dot b]_\infty (T- r)} -1)U^\star_T
:=e^{[\dot b]_\infty (T-r)}\Phi_1(r)U^\star_T.
\end{align}
Hence,  it follows that for all  $r\in[0,T]$, we have
$$
\EE\left[  \exp\left\{    \frac{ \rho }{(1-q)^2}|\D_r^jU^\star_{(1)}|^2 \right\}\right]\leq 
\EE\left[  \exp\left\{    \frac{ \rho e^{2[\dot b]_\infty (T-r)} \Phi^2_1(r)}{ (1-q)^2} |U^\star_{T}|^2\right\}\right].
$$
As assumption ${\mathbf (\mathbf R\mathbf 1)}$ is satisfied under  ${\mathbf (\mathbf R\mathbf 2)}$, then Theorem \ref{laplace1} applies and 

\begin{align}\label{DUB:1}
\forall \rho\in[0,&\frac{\rho_{\eqref{rho:lap1}} (1-q)^2n^2}{ e^{2[\dot b]_\infty (T-r)} \Phi^2_1(r)} ], \notag\\ 
&\EE^{ (1-q)}\left[  \exp\left\{    \frac{ \rho }{ (1-q)^2}|\D_r^jU^\star_{(1)}|^2 \right\}\right]\leq \exp\left
\{\rho \frac{e^{2[\dot b]_\infty (T-r)} \Phi^2_1(r)C_{\eqref{majufin}}(0)}{ (1-q)}\times\frac{(m-1)T^2}{ mn^2}\right\}.
\end{align}
\\
$\bullet\;${\bf Second term}. 
 By the second assertion of Lemma \ref{bbinv}, we have
$$
D_r^jU^\star_{(2)}\leq  e^{[\dot b]_\infty (T-r)}\int_{ {\hat\eta_{{n}}(r)}   }^{T}\gamma(t) 
 \left( \sqrt{d} [\ddot b]_\infty |b(X^{mn}_{\eta_{mn}(t)})|+\frac{1}{2} \|\nabla[\Delta b(X^{mn}_t)]\|
+[\dot b]^2_\infty\right)   dt.
$$
Moreover, thanks to Lemma \ref{momXn}, we get
\begin{align}
D_r^jU^\star_{(2)}&\leq e^{[\dot b]_\infty (T-r)} \notag\\ &\times \int_{ {\hat\eta_{{n}}(r)}   }^{T}\gamma(t)  \left( e^{[\dot b]_\infty t}(\sqrt{d} [\dot b]_\infty[\ddot b]_\infty +a_{\nabla\Delta b}) \left(\sup_{s\le t}|W_s|+\frac{|b(x_0)|}{ [\dot b]_\infty}\right)+
 a_{\nabla\Delta b}+[\dot b]^2_\infty\right)   dt \notag\\
&\leq C_{\eqref{majDu3}}e^{[\dot b]_\infty (T-r)}\int_{ {\hat\eta_{{n}}(r)}   }^{T}\gamma(t) e^{[\dot b]_\infty t}
 \left(  \sup_{s\le t}|W_s|+\frac{|b(x_0)|}{ [\dot b]_\infty}+1\right)   dt,\label{majDu3}
 \end{align}
  where 
$C_{\eqref{majDu3}}=
(\sqrt{d} [\dot b]_\infty[\ddot b]_\infty\vee[\dot b]^2_\infty +a_{\nabla\Delta b} 
) $.
With \eqref{majodu0}, we deduce that
\begin{align*}
\D_r^jU^\star_{(0)}+ D_r^jU^\star_{(2)} &\le 
  C_{\eqref{majDu3}}e^{[\dot b]_\infty (T-r)}\times \int_{ {\hat\eta_{{n}}(r)}   }^{T}\gamma(t)e^{[\dot b]_\infty t}\bigg(
  \sup_{s\le t}|W_s| \\&+\frac{\sqrt{t}}{\int_{ {\hat\eta_{{n}}(r)}   }^T\gamma(s)e^{[\dot b]_\infty s}\sqrt{s}ds}
  \bigg( \frac{|b(x_0)|+[\dot b]_\infty}{[\dot b]_\infty}\int_{ {\hat\eta_{{n}}(r)}   }^T\gamma(s)e^{[\dot b]_\infty s}ds+\frac{[\dot b]_\infty (m-1)T}{C_{\eqref{majDu3}}mn}\bigg)\bigg)dt.
\end{align*}
Therefore, using 
 Jensen's inequality for the probability density $p(t)=\frac{\sqrt{t}\gamma(t)e^{[\dot b]_\infty t}}{\int_{ {\hat\eta_{{n}}(r)}   }^T\sqrt{s}\gamma(s)e^{[\dot b]_\infty s}ds}$ on $[{ {\hat\eta_{{n}}(r)}   },T]$, we obtain that
\begin{align*}
   \EE&\left[  \exp\left\{\frac{ \rho }{q^2(1-\bar q)^2}|\D_r^jU^\star_{(0)}+\D_r^jU^\star_{(2)}|^2 \right\}\right]\notag\\&\le\int_{ {\hat\eta_{{n}}(r)}   }^T\EE\bigg[\exp\bigg\{\frac{\rho C^2_\eqref{majDu3}e^{2[\dot b]_\infty (T-r)}(\int_{ {\hat\eta_{{n}}(r)}   }^T\gamma(s)e^{[\dot b]_\infty s}\sqrt{s}ds)^2}{q^2(1-\bar q)^2} \times\bigg(   \frac{1}{\sqrt{t}} \sup_{s\le t}|W_s| +\delta\bigg)^2\bigg\}\bigg]p(t)dt.
\end{align*}
where $\delta=\frac{1}{\int_{ {\hat\eta_{{n}}(r)}   }^T\gamma(s)e^{[\dot b]_\infty s}\sqrt{s}ds}
  \bigg( \frac{|b(x_0)|+[\dot b]_\infty}{[\dot b]_\infty}\int_{ {\hat\eta_{{n}}(r)}   }^T\gamma(s)e^{[\dot b]_\infty s}ds+\frac{[\dot b]_\infty (m-1)T}{C_{\eqref{majDu3}}mn}\bigg)$ and by the scaling property for the Brownian motion $W$,  we may replace $\frac{1}{\sqrt{t}} \sup_{s\le t}|W_s|$ by $\sup_{s\le 1}|W_s|$.

In the same way as we did for the second term of Section \ref{secu}, we  use  that 
$
   \left(\sup_{s\le 1}|W_s|+\delta\right)^2\le \sum_{i=1}^d \left(\sup_{s\le 1}|W^i_s|+\delta\right)^2-(d-1)\delta^2,
$
to get 
\begin{align*}
   \EE&\left[  \exp\left\{\frac{ \rho }{q^2(1-\bar q)^2}|\D_r^jU^\star_{(0)}+\D_r^jU^\star_{(2)}|^2 \right\}\right]\\ &\le\exp\bigg\{-\frac{\rho(d-1) C^2_\eqref{majDu3}e^{2[\dot b]_\infty (T-r)}}{q^2(1-\bar q)^2}\times\bigg( \frac{|b(x_0)|+[\dot b]_\infty}{[\dot b]_\infty}\int_{ {\hat\eta_{{n}}(r)}   }^T\gamma(s)e^{[\dot b]_\infty s}ds+\frac{[\dot b]_\infty (m-1)T}{C_{\eqref{majDu3}}mn}\bigg)^2\bigg\}\notag\\&
   \times \EE^d\bigg[\exp\bigg\{\frac{\rho C^2_\eqref{majDu3}e^{2[\dot b]_\infty (T-r)}(\int_{ {\hat\eta_{{n}}(r)}   }^T\gamma(s)e^{[\dot b]_\infty s}\sqrt{s}ds)^2}{q^2(1-\bar q)^2}\times\bigg(   \sup_{s\le 1}|W^1_s|+\delta\bigg)^2\bigg\}\bigg].
\end{align*}
Applying the first assertion in Lemma \ref{Laplace_G_square} with $|H|=1$ and using that since $s\mapsto \sqrt{s}e^{[\dot b]_\infty s}$ and $s\mapsto e^{[\dot b]_\infty s}$ are non-decreasing, $\int_{ {\hat\eta_{{n}}(r)}   }^T\gamma(s)e^{[\dot b]_\infty s}\sqrt{s}ds\leq \frac{T(m-1)}{2mn}\int_{r}^Te^{[\dot b]_\infty s}\sqrt{s}ds$ and $\int_{ {\hat\eta_{{n}}(r)}   }^T\gamma(s)e^{[\dot b]_\infty s}ds\leq \frac{T(m-1)}{2mn}\times\frac{e^{[\dot b]_\infty T}-e^{[\dot b]_\infty r}}{[\dot b]_\infty}$, we deduce that if
$$
\rho\in[0, \frac{\rho_{\eqref{DUB:3}} q^2(1-\bar q)^2n^2}{e^{2[\dot b]_\infty (T-r)} \Phi_2^2(r)}] \mbox{ with }
\rho_{\eqref{DUB:3}}:=  \frac{m^2}{2C^2_{\eqref{majDu3}} T^2(m-1)^2} \mbox{ and }   \Phi_2(r):=\int_{r}^Te^{[\dot b]_\infty s}\sqrt{s}ds,
$$
then
\begin{align}\label{DUB:3}
  &\EE^{{q(1-\bar q)} }\left[  \exp\left\{\frac{ \rho}{ q^2(1-\bar q)^2}| \D_r^jU^\star_{(0)}+D_r^jU^\star_{(2)}|^2 \right\}\right]\le\exp\left\{\rho\frac{  e^{2[\dot b]_\infty (T-r)} \phi_2(r,[\dot b]_\infty)}{q(1-\bar q)}\times\frac{  T^2(m-1)}{mn^2}\right\}
\end{align}
where  
\begin{align*}
\phi_2(r,x)&=\frac{(m-1)}{m}\bigg( (3d+1)\left(C_{\eqref{majDu3}}\frac{|b(x_0)|+[\dot b]_\infty}{2[\dot b]^2_\infty}(e^{[\dot b]_\infty T}-e^{[\dot b]_\infty r})+x\right)^2\\&+
\frac{4dC_{\eqref{majDu3}}\Phi_2(r)}{\sqrt{\pi}}\left(C_{\eqref{majDu3}}\frac{|b(x_0)|+[\dot b]_\infty}{2[\dot b]^2_\infty}(e^{[\dot b]_\infty T}-e^{[\dot b]_\infty r})+x\right)+d\ln 2C^2_{\eqref{majDu3}}\Phi^2_2(r)\bigg).
\end{align*}
\\
$\bullet\;${\bf Third term}. 
Let us introduce the quantities  
\begin{align}
   C&=\sqrt{\frac{e^{2[\dot b]_\infty(T-\hat\eta_{n}(r))}-1}{2[\dot b]_\infty}}+\sum_{k={\lceil \frac{rn
}{T}\rceil }+1}^{n-1}\frac{T[\dot b]_\infty  e^{[\dot b]_\infty(T-t_{k+1})}\sqrt{(e^{2[\dot b]_\infty(t_k-\hat\eta_{n}(r))}-1)/2[\dot b]_\infty}}{n}\notag
   \\ &\le {e^{[\dot b]_\infty(T-r)}}\left(\sqrt{\frac{1-e^{-2[\dot b]_\infty(T-r)}}{2[\dot b]_\infty}}+\sqrt{\frac{[\dot b]_\infty}{2}}\int_r^T\sqrt{1-e^{-2[\dot b]_\infty(t-r)}}dt \right):=e^{[\dot b]_\infty(T-r)}\Phi_3(r),\label{cteD:maj}
\end{align}
 $p_n=\frac{1}{C}\sqrt{\frac{e^{2[\dot b]_\infty(T-\hat\eta_{n}(r))}-1}{2[\dot b]_\infty}}$ and $p_{k}=\frac{T[\dot b]_\infty  e^{[\dot b]_\infty(T-t_{k+1})}\sqrt{(e^{2[\dot b]_\infty(t_k-\hat\eta_{n}(r))}-1)/2[\dot b]_\infty}}{Cn}$ for ${\lceil \frac{rn
}{T}\rceil }+1\le k\le n-1$. 
Notice that $\sum_{k=\lceil \frac{rn
}{T}\rceil +1}^{n}p_{k}=1$ so that we have defined a probability measure. Therefore, by \eqref{DU2} we have
\begin{align*}
\D^j_r  U_{(3)}^{\star}&= \sum_{k={\lceil \frac{rn
}{T}\rceil }+1}^{n-1}    p_k\frac{ C\sqrt{2[\dot b]_\infty}}{\sqrt{e^{2[\dot b]_\infty(t_k-\hat\eta_{n}(r))}-1}}\left |\int_{ {\hat\eta_{{n}}(r)}   }^{t_{k}}\gamma(s)\nabla^2 b(X^{mn}_s)\D^j_rX^{mn}_{{_{\eta_{{n}}(s)} } }dW_s\right |\notag\\
&+p_n\frac{ C\sqrt{2[\dot b]_\infty}}{\sqrt{e^{2[\dot b]_\infty(T-\hat\eta_{n}(r))}-1}}\max_{{\lceil \frac{rn
}{T}\rceil +1} \le k\le n} \left|\int_{ {\hat\eta_{{n}}(r)}   }^{t_{k}}\gamma(s)\nabla^2 b(X^{mn}_s)\D^j_rX^{mn}_{{_{\eta_{{n}}(s)} } }dW_s\right|
\end{align*}
and so
$$
\D^j_r  U_{(3)}^{\star}\le \sum_{k= {\lceil \frac{rn
}{T}\rceil }+1}^{n}p_k\frac{ C\sqrt{2[\dot b]_\infty}}{\sqrt{e^{2[\dot b]_\infty(t_k-\hat\eta_{n}(r))}-1}}
\max_{{\lceil \frac{rn}{T}\rceil }+1 \le l\le k}\left |\int_{ {\hat\eta_{{n}}(r)}  }^{t_{l}}\gamma(t)\nabla^2 b(X^{mn}_s)\D^j_rX^{mn}_{{_{\eta_{{n}}(s)} } }dW_s\right|.
$$
Now, applying Jensen's inequality to the convex function $\R\ni x\mapsto \exp\left\{\frac{\rho x^2}{q^2\bar q^2}\right\}$, we deduce that  for all 
$r\in[0,T]$, we have
\begin{align*}
  & \EE \left[  \exp\left\{\frac{ \rho }{ q^2\bar q^2}|\D_r^jU^\star_{(3)}|^2 \right\}\right]
\le\sum_{k= {\lceil \frac{rn
}{T}\rceil }+1}^{n} p_k\\ &\times \EE\bigg[ \exp\left\{
\frac{ \rho e^{2[\dot b]_\infty(T-r)}\Phi^2_3(r)2[\dot b]_\infty}{q^2\bar q^2 (e^{2[\dot b]_\infty(t_k-\hat\eta_{n}(r))}-1) }
\max_{{\lceil \frac{rn
}{T}\rceil }+1 \le l\le k}\left |\int_{ {\hat\eta_{{n}}(r)}  }^{t_{l}}\gamma(t)\nabla^2 b(X^{mn}_s)\D^j_rX^{mn}_{{_{\eta_{{n}}(s)} } }dW_s\right|^2
\right\}\bigg].
\end{align*}
Now, using Assumption ${\mathbf (\mathbf R\mathbf 2)}$, Lemma \ref{boundMal}, than the periodicity of the function $\gamma$ with period $t_1=T/n$,
 we get
\begin{align*}
\int_ {\hat\eta_{{n}}(r)}^{t_{k}}&\gamma(t)^2\Tr\Bigl[\nabla ^2b(X^{mn}_t) \D^j_rX^{mn}_{{_{\eta_{{n}}(t)} } }
(\nabla^2 b(X^{mn}_t)\D^j_rX^{mn}_{{_{\eta_{{n}}(t)} } })^{\tp}\Bigr]dt\\
&\le \int_ {\hat\eta_{{n}}(r)}^{t_{k}}\gamma(t)^2\sum_{j=1}^d  \left \|{\frac{\partial\nabla b(X^{mn}_t)}{\partial x_j} }\right \|^2|{\D^j_rX^{mn}_{{_{\eta_{{n}}(t)} } }}|^2dt\le d [\ddot b]^2_\infty\int_ {\hat\eta_{{n}}(r)}^{t_{k}}\gamma(t)^2e^{2[\dot b]_\infty(\eta_{{n}}(t)-\hat\eta_{{n}}(r)) }dt\\
&=d [\ddot b]^2_\infty T^2\frac{(2m-1)(m-1)}{6(mn)^2}\frac{T}{n}\sum_{l=\lceil\frac{nr}{T}\rceil}^{k-1}e^{2[\dot b]_\infty(t_l-\hat\eta_{{n}}(r)) }\\&\le d [\ddot b]^2_\infty T^2\frac{(2m-1)(m-1)}{6(mn)^2}\times\frac{e^{2[\dot b]_\infty(t_k-\hat\eta_{{n}}(r)) }-1}{2[\dot b]_\infty}.
\end{align*}
Then, by the second assertion in Lemma \ref{Laplace_G_square},
\begin{align}
 &\forall \rho\in [0,\frac{\rho_{\eqref{DUB:2}}q^2\bar q^2n^2}{e^{2[\dot b]_\infty(T-r)}\Phi_3^2(r)}], \mbox{ with } \rho_{\eqref{DUB:2}}:=\frac{3m^2}{4 T^2 d[\ddot b]^2_\infty (2m-1)(m-1) }\notag\\
  &  \EE^{q\bar q}\left[  \exp\left\{\frac{ \rho }{ q^2\bar q^2 }|D_r^jU^\star_{(3)}|^2 \right\}\right]\le
 \exp\left\{\rho \frac{e^{2[\dot b]_\infty(T-r)}\Phi_3^2(r)C_\eqref{DUB:2}}{q\bar q}\times\frac{(m-1)T^2}{mn^2 }\right\},
\label{DUB:2}
\end{align}
where $C_{\eqref{DUB:2}}:=\frac{2}{3}\ln(2) d[\ddot b]^2_\infty \frac{2m-1}{m} $.\\\\
$\bullet\,$ {\bf Conclusion}.
In order to have the same constraint on $\rho$ for the three Laplace transforms, we choose 
\begin{equation*}
  q  =\frac{\frac{\sqrt{\rho_{\eqref{rho:lap1}}}}{\Phi_1(r)}({\frac{\sqrt{\rho_{\eqref{DUB:3}}}}{\Phi_2(r)}+\frac{\sqrt{\rho_{\eqref{DUB:2}}}}{\Phi_3(r)})}}{ \frac{\sqrt{\rho_{\eqref{rho:lap1}}}}{\Phi_1(r)}{\frac{\sqrt{\rho_{\eqref{DUB:3}}}}{\Phi_2(r)}+\frac{\sqrt{\rho_{\eqref{rho:lap1}}}}{\Phi_1(r)}\frac{\sqrt{\rho_{\eqref{DUB:2}}}}{\Phi_3(r)}}+ {\frac{\sqrt{\rho_{\eqref{DUB:3}}}}{\Phi_2(r)}\frac{\sqrt{\rho_{\eqref{DUB:2}}}}{\Phi_3(r)}}}
   \mbox{ and }
 \bar q= \frac{\frac{\sqrt{\rho_{\eqref{DUB:3}}}}{\Phi_2(r)}}{\frac{\sqrt{\rho_{\eqref{DUB:3}}}}{\Phi_2(r)}+\frac{\sqrt{\rho_{\eqref{DUB:2}}}}{\Phi_3(r)}}
\end{equation*}
Then, by combining \eqref{Laplace_DU_2}, \eqref{DUB:1}, \eqref{DUB:3}  and 
\eqref{DUB:2}, we deduce that  if 
$$
0\le \rho\leq \frac{\frac{{\rho_{\eqref{rho:lap1}}}}{\Phi^2_1(r)}\frac{{\rho_{\eqref{DUB:3}}}}{\Phi^2_2(r)}\frac{{\rho_{\eqref{DUB:2}}}}{\Phi^2_3(r)}}
{e^{2[\dot b]_\infty(T-r)}\left(\frac{\sqrt{\rho_{\eqref{rho:lap1}}}}{\Phi_1(r)}{\frac{\sqrt{\rho_{\eqref{DUB:3}}}}{\Phi_2(r)}+\frac{\sqrt{\rho_{\eqref{rho:lap1}}}}{\Phi_1(r)}\frac{\sqrt{\rho_{\eqref{DUB:2}}}}{\Phi_3(r)}}+ {\frac{\sqrt{\rho_{\eqref{DUB:3}}}}{\Phi_2(r)}\frac{\sqrt{\rho_{\eqref{DUB:2}}}}{\Phi_3(r)}}\right)^2}\times n^2$$
then as
\begin{multline*}
\Phi(r,[\dot b]_\infty)= \bigg({ \frac{\sqrt{\rho_{\eqref{rho:lap1}}}}{\Phi_1(r)}{\frac{\sqrt{\rho_{\eqref{DUB:3}}}}{\Phi_2(r)}+\frac{\sqrt{\rho_{\eqref{rho:lap1}}}}{\Phi_1(r)}\frac{\sqrt{\rho_{\eqref{DUB:2}}}}{\Phi_3(r)}}+ {\frac{\sqrt{\rho_{\eqref{DUB:3}}}}{\Phi_2(r)}\frac{\sqrt{\rho_{\eqref{DUB:2}}}}{\Phi_3(r)}}}\bigg)\\\times\bigg(\frac{\Phi^2_3(r)C_\eqref{DUB:2} }{\frac{\sqrt{\rho_{\eqref{rho:lap1}}}}{\Phi_1(r)} \frac{\sqrt{\rho_{\eqref{DUB:3}}}}{\Phi_2(r)}}+\frac{\phi_2(r,[\dot b]_\infty)}{\frac{\sqrt{\rho_{\eqref{rho:lap1}}}}{\Phi_1(r)} \frac{\sqrt{\rho_{\eqref{DUB:2}}}}{\Phi_3(r)}}+
\frac{\Phi^2_1(r)C_{\eqref{majufin}}(0)}{{\frac{\sqrt{\rho_{\eqref{DUB:3}}}}{\Phi_2(r)}\frac{\sqrt{\rho_{\eqref{DUB:2}}}}{\Phi_3(r)}}}\bigg)
\end{multline*}
we get
\begin{align*}
\EE&\left[  \exp\left\{ \rho|\D^j_rU_T^{\star}|^2\right\}\right]\leq
 \exp\left\{\rho e^{2[\dot b]_\infty(T-r)}\Phi(r,[\dot b]_\infty)\frac{(m-1)T^2}{mn^2}\right\}.
\end{align*}
\end{proof}
\section{Proof of Theorem \ref{thm:Orlicz}}\label{sec:6}
For $\lambda \in\RR$, by independence, \begin{align}\label{decomplap}
   &\EE\left[\exp\left(\lambda[\hat Q-\mathbb Ef(X^{m^L}_T)]\right)\right]=\prod_{\ell=0}^L\EE\left[\exp(\lambda \hat Q_{\ell})\right]\mbox{ where}\\
&\EE\left[\exp(\lambda \hat Q_{\ell})\right]=\left( \EE\left[\exp\left\{\frac{\lambda}{N_{\ell}}\left(f(X^{m^{\ell}}_T)-f(X^{m^{\ell-1}}_T)-
\EE[f(X^{m^{\ell}}_T)-f(X^{m^{\ell-1}}_T)]\right)\right\}\right]\right)^{N_{\ell}}\notag\\
\mbox{ and }&\EE\left[\exp(\lambda \hat Q_{0})\right]\leq \exp\left\{ \frac{\lambda^2 T[\dot f]_\infty^2}{2N_0}\right\},\notag\end{align}
where we used the Gaussian concentration bound  \eqref{Co:gau} to get the last the inequality.
For $\ell\in\{1,\hdots,L\}$, we  set $n\in\NN^*$ and define 
$$\Upsilon:=f(X^{mn}_T)-f(X^{n}_T)-
\EE[f(X^{mn}_T)-f(X^{n}_T)].$$
For $\tilde{\lambda}\in\RR$, we want to obtain an estimation of $\EE\left[\exp(\tilde{\lambda}\Upsilon)\right]$  of the form $\exp\left\{C\tilde{\lambda}^2\frac{(m-1)T^2}{mn^2}\right\}$  where $C$ is an explicit constant and $\frac{(m-1)T^2}{mn^2}$ is the order of the variance of the centered random variable $\Upsilon$ according to Proposition \ref{properfort}. To do so, we assume that $f\in \mathscr C^1_b(\RR^d,\RR)$ is Lipschitz continuous with constant $[\dot f]_{\infty}$ and such that $\nabla f$ is also Lipschitz with constant $[\dot f]_{\rm lip}$.
By Clarck's Ocone formula we have 
$$ f(X^{mn}_T)-f(X^{n}_T)-
\EE[f(X^{mn}_T)-f(X^{n}_T)]=
\int_0^T\EE\left[K_{r}|\mathcal F_r\right]\cdot dW_r,$$
where, for $j\in\{1,\dots,d\}$, the $j^{{th}}$-component of the $d$-dimensional vector $K_{r}$ is given by 
$K_{r,j}:=\D^j_rf(X^{mn}_T)-\D^j_rf(X^{n}_T)$. 
For $p\in(0,1)$, we use H\"older's inequality to get
\begin{multline*}
\EE\left[\exp(\tilde \lambda \Upsilon)\right]\leq 
\EE^{p}\left[\exp\left\{\frac{\tilde \lambda}{p}
\int_0^T\EE\left[K_{r}|\mathcal F_r\right]\cdot dW_r -\frac{\tilde\lambda^2}{2p^2}\int_0^T |\EE\left[K_{r}|\mathcal F_r\right]|^2dr
\right\}\right]\\\times \EE^{1-p}\left[\exp\left\{\frac{\tilde \lambda^2}{2p(1-p)}\int_0^T |\EE\left[K_{r}|\mathcal F_r\right]|^2dr
\right\}\right]. 
\end{multline*}
Now, by  the Malliavin chain rule we have

\begin{equation}\label{chainrule}
 K_{r}=\D_r X^{mn}_{T}\nabla f(X^{mn}_T ) -\D_r X^{n}_{T} \nabla f(X^{n}_T ) ,
\end{equation}
where $\D_r X^{n}_{T}=(\D^i_r X^{n}_{T,j})_{1\leq i,j\leq d}\in \RR^{d\times d}$.

According to Lemma \ref{boundMal} and under our assumption on the boundedness of $\nabla f$, we easily check that  $\sup_{r\in[0,T]}|K_{r}|^2\leq 4d e^{2T [\dot b]_{\infty}} [\dot f]^2_{\infty} $. Therefore, the process
$$\left(\exp\left\{\frac{\tilde \lambda}{p}
\int_0^t\EE\left[K_{r}|\mathcal F_r\right]\cdot dW_r -\frac{\tilde\lambda^2}{2p^2}\int_0^t |\EE\left[K_{r}|\mathcal F_r\right]|^2dr
\right\}\right)_{0\leq t\leq T}$$ is a martingale, which together with the choice $p=1/2$ which minimizes $\frac{1}{p(1-p)}$ leads us to  
\begin{equation*}
\EE\left[\exp(\tilde \lambda \Upsilon)\right]\leq \EE^{1/2}\left[\exp\left\{2\tilde \lambda^2
 \int_0^T|\EE\left[K_{r}|\mathcal F_r\right]|^2dr\right\}\right]. 
\end{equation*}
Applying Jensen's inequality twice, and now denoting by $p$ a measurable positive function such that $\int_0^T p(r)dr=1$, we obtain that
\begin{align*}
   \EE&\left[  \exp\left\{ 2\tilde \lambda^2\int_0^{T} |\EE\left[K_{r}|\mathcal F_r\right]|^2dr \right\}\right]\leq\int_0^T \EE\left[  \exp\left\{ \frac{2\tilde \lambda^2T}{p(r)}|K_{r}|^2 \right\}\right]p(r)dr.\end{align*}
and deduce that 
\begin{equation}
   \EE\left[\exp(\tilde \lambda \Upsilon)\right]\leq \left(\int_0^T \EE\left[  \exp\left\{ \frac{2\tilde \lambda^2T}{p(r)}|K_{r}|^2 \right\}\right]p(r)dr\right)^{1/2}.\label{lapgam}
\end{equation}
We now want to estimate the moment generating function of $|K_r|^2$. Setting $$U_T:=X^{mn}_T-X^{n}_T,$$ remarking that $\| \D_r U_{T}\|^2\leq \Tr\left[\D_r U_{T}(\D_r U_{T})\tp\right]=\sum_{j=1}^d|\D_r^j U_T|^2 $ and, by Lemma \ref{boundMal}, $\|D_rX^{mn}_T\|\le \sqrt{d}e^{[\dot b]_{\infty}(t-r)}$, we obtain that
\begin{align}\label{majka}
   |K_r|&\le \|D_rX^{mn}_T\||\nabla f(X^{mn}_T )-\nabla f(X^{n}_T )|+\|\D_r X^{mn}_{T}- \D_r X^{n}_{T}\||\nabla f(X^{n}_T )|\notag \\&\le \sqrt{d}[\dot f]_{\rm{lip}}e^{[\dot b]_{\infty}(t-r)}|U_T|+[\dot f]_{\infty}\left(\sum_{j=1}^d|\D_r^j U_T|^2 \right)^{1/2}.
\end{align}
A careful look at the proof of this theorem shows that, in the decomposition \eqref{decdju} of $\D^j_rU_T$, the sum $\D_r^j U^{(1-3)}_T:=\D_r^j U^{(1)}_T+\D_r^j U^{(2)}_T+\D_r^j U^{(3)}_T$  goes to $0$ as $r\to T$ whereas $\D_r^jU^{(0)}_{T}$ does not. This indicates that it is not optimal to combine $\D_r^jU^{(0)}_{T}$ with $\D_r^j U^{(2)}_T$ as in this proof. We also notice that under the same constraint on $\rho$ as in the theorem, 
\begin{align}\label{Th_Laplace_DU}
\EE&\left[  \exp\left\{ \rho |\D_r^j U^{(1-3)}_T|^2\right\}\right]\leq
 \exp\left\{\rho e^{2[\dot b]_\infty(T-r)}\Phi(r,0)\frac{(m-1)T^2}{mn^2}\right\}.
\end{align}
Since $U_T$ does not depend on $r$, it should be better to combine $\D^jU^{(0)}_{T}$ with it by replacing \eqref{majka} by the estimation 
\begin{align*}
   |K_r|&\le \sqrt{d}[\dot f]_{\rm{lip}}e^{[\dot b]_{\infty}(T-r)}|U_T|+[\dot f]_{\infty}\left(\sum_{j=1}^d|\D_r^j U^{(0)}_T|^2 \right)^{1/2}+[\dot f]_{\infty}\left(\sum_{j=1}^d|\D_r^j U^{(1-3)}_T|^2 \right)^{1/2}\\
&\le \sqrt{d}[\dot f]_{\rm{lip}}e^{[\dot b]_{\infty}(T-r)}\left(\frac{(m-1)T^2x}{2mn}+|U_T|\right) +[\dot f]_{\infty}\left(\sum_{j=1}^d|\D_r^j U^{(1-3)}_T|^2 \right)^{1/2}\\&\mbox{ where } x=\frac{2[\dot b]_{\infty}[\dot f]_{\infty}}{T[\dot f]_{\rm{lip}}}
\end{align*}
which takes \eqref{term_ini} into account. 
One deduces that for $\kappa(r)\in (0,1)$, 
$$|K_{r}|^2
\leq \frac{1}{\kappa(r)}d[\dot f]_{\rm{lip}}^2e^{2[\dot b]_{\infty}(T-r)}\left(\frac{(m-1)T^2x}{2mn}+|U_T|\right)^2 +\frac{[\dot f]_{\infty}^2}{(1-\kappa(r))}\sum_{j=1}^d|\D_r^j U^{(1-3)}_T|^2.$$
Combining \eqref{lapgam}, H\"older's inequality and the convexity
of the exponential function which ensures that the exponential of the mean of $d$ terms is not  greater than the mean of the exponentials, we deduce that
\begin{multline}\label{split_error}
 \EE^2\left[\exp(\tilde \lambda \Upsilon)\right]\leq\int_0^T\EE^{\kappa(r)}\left[\exp\left\{\frac{2d\tilde\lambda^2 T[\dot f]^2_{_{\rm{lip}}} e^{2[\dot b]_{\infty}(T-r) }  }{p(r)\kappa^2(r)}
\left(\frac{(m-1)T^2x}{2mn}+|U_T|\right)^2 \right\}\right]\\
\times \left(\frac{1}{d}\sum_{j=1}^d\EE\left[\exp\left\{\frac{2d\tilde\lambda^2T[\dot f]^2_{\infty}}{p(r)(1-\kappa(r))^2}
|\D^j_r U^{(1-3)}_{T}|^2 \right\}\right]\right)^{1-\kappa(r)}p(r)dr.
\end{multline}
We now choose $\kappa(r)=\frac{[\dot f]_{_{\rm{lip}}}\sqrt{\hat\rho(r)}}{[\dot f]_{\infty}\sqrt{\rho_{\eqref{rho:lap1}}}+[\dot f]_{_{\rm{lip}}}\sqrt{\hat\rho(r)}}$ to obtain the same constraint on $\tilde \lambda$ for the two expectations at time $r$ and then $p(r)\propto \frac{e^{2[\dot b]_{\infty}(T-r) }}{\hat\rho(r)}\left([\dot f]_{\infty}\sqrt{\rho_{\eqref{rho:lap1}}}+[\dot f]_{_{\rm{lip}}}\sqrt{\hat\rho(r)}\right)^2$ to ensure that this common constraint does not depend on $r$. Notice that since the functions $\Phi_i$ are continuous on $[0,T]$, positive on $[0,T)$ and such that $\Phi_1(r)={\mathcal O}(T-r)$, $\Phi_2(r)={\mathcal O}(T-r)$ and $\Phi_3(r)={\mathcal O}(\sqrt{T-r})$ so that $\hat{\rho}(r)={\mathcal O}((T-r)^{-1})$ as $r\to T-$, the function $p$ is bounded and therefore integrable on $[0,T]$.
We conclude that if
\begin{align*}
   \tilde \lambda^2\le \frac{\rho_{\eqref{rho:lap1}}n^2}{2dT\int_0^T\frac{e^{2[\dot b]_{\infty}(T-t) }}{\hat\rho(t)}\left([\dot f]_{\infty}\sqrt{\rho_{\eqref{rho:lap1}}}+[\dot f]_{_{\rm{lip}}}\sqrt{\hat\rho(t)}\right)^2dt},
\end{align*}
\begin{align}
  \EE^2\left[\exp(\tilde \lambda \Upsilon)\right]&\leq \int_0^T\exp\bigg\{\frac{2d\tilde \lambda^2 T([\dot f]_{_{\rm{lip}}}\sqrt{\hat\rho(r)}C_{\eqref{majufin}}({2[\dot b]_{\infty}[\dot f]_{\infty}}/{T[\dot f]_{\rm{lip}}})+[\dot f]_{\infty}\hat\rho(r)\Phi(r,0)/\sqrt{\rho_{\eqref{rho:lap1}}})}{[\dot f]_{\infty}\sqrt{\rho_{\eqref{rho:lap1}}}+[\dot f]_{_{\rm{lip}}}\sqrt{\hat\rho(r)}}\notag\\&\phantom{\leq \int_0^T\exp\bigg\{}\times\int_0^T\frac{e^{2[\dot b]_{\infty}(T-t) }}{\hat\rho(t)}\left([\dot f]_{\infty}\sqrt{\rho_{\eqref{rho:lap1}}}+[\dot f]_{_{\rm{lip}}}\sqrt{\hat\rho(t)}\right)^2dt\bigg\}p(r)dr\notag\\
&\le \exp\left\{\frac{2\tilde \lambda^2[\dot f]^2_{\infty} C_{\eqref{majofin}} (m-1)T^2}{mn^2}\right\},\label{majofin}
\end{align}
where 
\begin{align*}
   C_{\eqref{majofin}}=&d T\int_0^T\frac{e^{2[\dot b]_{\infty}(T-t) }}{\hat\rho(t)}\left(\sqrt{\rho_{\eqref{rho:lap1}}}+\frac{[\dot f]_{_{\rm{lip}}}}{[\dot f]_{\infty}}\sqrt{\hat\rho(t)}\right)^2dt\\&\times\sup_{r\in [0,T)}\frac{[\dot f]_{_{\rm{lip}}}\sqrt{\hat\rho(r)}C_{\eqref{majufin}}({2[\dot b]_{\infty}[\dot f]_{\infty}}/{T[\dot f]_{\rm{lip}}})+[\dot f]_{\infty}\hat\rho(r)\Phi(r,0)/\sqrt{\rho_{\eqref{rho:lap1}}}}{[\dot f]_{\infty}\sqrt{\rho_{\eqref{rho:lap1}}}+[\dot f]_{_{\rm{lip}}}\sqrt{\hat\rho(r)}}
\end{align*}
is finite since, as $r\to T-$, $\phi_2(r,0)={\mathcal O}((T-r)^{2})$ and $\Phi(r,0)={\mathcal O}((T-r))$. We complete the proof using \eqref{decomplap}. 
\appendix
\section{Proofs of the technical lemmas}
\begin{lemma}\label{Laplace_G_square}
   Let $(H_t)_{t\leq T}$ be an adapted $\R^d$-valued process and $|H|:=\left|\left(\int_0^T|H_t|^2dt\right)^{1/2}\right|_\infty$. Then $\forall \delta\ge 0$, $\forall \mu\in\left[0,\frac{1}{4|H|^2}\right)$, 
\begin{align*}
  &\;\E\left(e^{\mu(\delta+\sup_{t\in[0,T]}|\int_0^t H_s.dW_s|)^2}\right)\leq\exp\left\{\frac{2\mu\delta^2}{1-4\mu|H|^2}\right\}\left(\frac{8\mu\delta|H|}{\sqrt{2\pi}(1-4\mu|H|^2)}+\frac{1}{\sqrt{1-4\mu|H|^2}}\right),\\
\end{align*}
where the right-hand side is smaller than $\exp\left\{4\mu\left(\delta^2+\frac{2}{\sqrt{\pi}}\delta|H|+|H|^2\ln 2\right)\right\}$ when moreover $\mu \in\left[0,\frac{1}{8|H|^2}\right]$.

For $(M(t))_{t\leq T}$ an adapted $\R^{d\times d}$-valued process and $|M|:=\left|\left(\int_0^T\Tr(M(t)M^*(t))dt\right)^{1/2}\right|_\infty$,
\begin{align*}
\forall \mu\in\left[0,\frac{1}{4|M|^2}\right),\;  &\E\left(e^{\mu\sup_{t\in[0,T]}|\int_0^t M_sdW_s|^2}\right)\leq\frac{1}{\sqrt{1-4\mu|M|^2}},\\
&\mbox{ and }\forall \mu\in\left[0,\frac{1}{8|M|^2}\right],\;\E\left(e^{\mu\sup_{t\in[0,T]}|\int_0^t M_sdW_s|^2}\right)\leq e^{4\mu|M|^2\ln 2} \end{align*}
\end{lemma}
\begin{remark}
Let $h(t)=\||H_t|\|_\infty$ for $t\in[0,T]$. Then
$$\forall \mu\in\left[0,\frac{1}{2\|h\|_2^2}\right],\;\sup_{t\in[0,T]}\E\left(e^{\mu(\int_0^t H_s.dW_s)^2}\right)\leq\frac{1}{\sqrt{1-2\mu\|h\|_2^2}},$$
where $\|h\|_2^2=\int_0^Th^2(t)dt\geq |H|^2$. 

When $\|h\|_2^2<+\infty$, this is a consequence of the convexity of $x\mapsto e^{\mu x^2}$ and Jensen's inequality. Indeed introducing (on a possibly enlarged probability space) $(\beta_t)_{t\in[0,T]}$ a one-dimensional Brownian motion independent from ${\mathcal F}^{W,H}=\sigma((H_t,W_t),t\in[0,T])$, we obtain that for $t\in [0,T]$, $\int_0^tH_s.dW_s+\int_0^t\sqrt{h^2(s)-|H_s|^2}d\beta_s$ is a centered Gaussian random variable with variance equal to $\int_0^th^2(s)ds$ such that $$\E\left(\int_0^tH_s.dW_s+\int_0^t\sqrt{h^2(s)-|H_s|^2}d\beta_s\bigg|{\mathcal F}^{W,H}\right)=\int_0^tH_s.dW_s.$$
\end{remark}
\begin{proof}
The argument is based on the Dambins-Dubins-Schwarz (see e.g. \cite{RevYor})  theorem which ensures the existence of a one-dimensional standard Brownian motion $(\beta_t)_{t\geq 0}$ such that $\forall t\geq 0$, $\int_0^tH_s.dW_s=\beta_{\int_0^t|H_s|^2ds}$. Hence
\begin{align*}
   \E\left(e^{\mu(\delta+\sup_{t\in[0,T]}|\int_0^t H_s.dW_s|)^2}\right)&=\E\left(e^{\mu\left(\delta+\sup_{t\in[0,T]}|\beta_{\int_0^t|H_s|^2ds}|\right)^2}\right)\leq \E\left(e^{\mu(\delta+\sup_{s\in[0,|H|^2]}|\beta_{s}|)^2}\right)
\\&\leq \E\left(e^{\mu(\delta+\sup_{s\in[0,|H|^2]}\beta_{s})^2}e^{\mu(\delta-\inf_{s\in[0,|H|^2]}\beta_{s})^2}\right)\\&
\leq \E^{1/2}\left(e^{2\mu(\delta+\sup_{s\in[0,|H|^2]}\beta_{s})^2}\right)\E^{1/2}\left(e^{2\mu(\delta-\inf_{s\in[0,|H|^2]}\beta_{s})^2}\right)\\&=\E\left(e^{2\mu(\delta+|\beta_{|H|^2}|)^2}\right)
\end{align*}
where we have used that $\sup_{s\in[0,|H|^2]}\beta_s$ and $-\inf_{s\in[0,|H|^2]}\beta_s
$ have the same law as $|\beta_{|H|^2}|$ for the last equality. Now, using the change of variables $y=x-4\mu\delta|H|^2/(1-4\mu|H|)$ for the second equality, we obtain that
\begin{align*}
   \E&\left(e^{2\mu(\delta+|\beta_{|H|^2}|)^2}\right)=2\int_0^\infty \exp\left\{2\mu(\delta+x)^2-\frac{x^2}{2|H|^2}\right\}\frac{dx}{|H|\sqrt{2\pi}}\\&=2\exp\left\{\frac{2\mu\delta^2}{1-4\mu|H|^2}\right\}\int_{-4\mu\delta|H|^2/(1-4\mu|H|^2)}^\infty \exp\left\{-\frac{(1-4\mu|H|^2)y^2}{2|H|^2}\right\}\frac{dy}{|H|\sqrt{2\pi}}\\
&\le \exp\left\{\frac{2\mu\delta^2}{1-4\mu|H|^2}\right\}\left(\frac{8\mu\delta|H|}{\sqrt{2\pi}(1-4\mu|H|^2)}+\int_{-\infty}^\infty \exp\left\{-\frac{(1-4\mu|H|^2)y^2}{2|H|^2}\right\}\frac{dy}{|H|\sqrt{2\pi}}\right)\\
&=\exp\left\{\frac{2\mu\delta^2}{1-4\mu|H|^2}\right\}\left(\frac{8\mu\delta|H|}{\sqrt{2\pi}(1-4\mu|H|^2)}+\frac{1}{\sqrt{1-4\mu|H|^2}}\right).
\end{align*}
The concavity of the logarithm ensures that $\forall x\in[0,\frac 12],\;\ln(1-x)\ge -2x\ln 2$ so that $\frac{1}{\sqrt{1-x}}=e^{-\frac 12\ln(1-x)}\le e^{x\ln 2}$.
Therefore when $\mu \in\left[0,\frac{1}{8|H|^2}\right]$, 
\begin{align*}
\exp\left\{\frac{2\mu\delta^2}{1-4\mu|H|^2}\right\}\left(\frac{8\mu\delta|H|}{\sqrt{2\pi}(1-4\mu|H|^2)}+\frac{1}{\sqrt{1-4\mu|H|^2}}\right)&\le e^{4\mu\delta^2}\left(\frac{8\mu\delta|H|}{\sqrt{\pi}}+1\right)\frac{1}{\sqrt{1-4\mu|H|^2}}\\
&\le e^{4\mu\delta^2}e^{\frac{8\mu\delta|H|}{\sqrt{\pi}}}e^{4\mu |H|^2\ln 2}.
\end{align*}

  Let now for $i\in\{1,\hdots,d\}$, $M_i(t)$ denote the $i$-th line of the matrix $M(t)$ and $|M_i|:=\left\|\left(\int_0^T\sum_{j=1}^d M_{ij}^2(t)dt\right)^{1/2}\right\|_\infty$. For $\mu< \frac{1}{4|M|^2}$, we have
\begin{align*}\E\left(e^{\mu\sup_{t\in[0,T]}|\int_0^t M(s)dW_s|^2}\right)&=\E\left(e^{\mu\sup_{t\in[0,T]}\sum_{i=1}^d(\int_0^t M_i(s)dW_s)^2}\right)\leq \E\left(e^{\mu\sum_{i=1}^d\sup_{t\in[0,T]}(\int_0^t M_i(s)dW_s)^2}\right)\\
&\leq \E\left(\sum_{i=1}^d \frac{|M_i|^2}{|M|^2}e^{\frac{\mu|M|^2}{|M_i|^2}\sup_{t\in[0,T]}(\int_0^t M_i(s)dW_s)^2}\right)\leq \frac{1}{\sqrt{1-4\mu|M|^2}}
\end{align*}
where we used Jensen's inequality for the third inequality and the first statement of the Lemma for the fourth.
\end{proof}
%\section*{Acknowledgements}
%{We would like to thank the referee for his valuable comments that helped us to improve the paper.}
%

\end{document}